\documentclass[11pt,reqno]{amsproc}

\title[]{Reconnection in 2D Inertial Magnetohydrodynamics}
\author{Lydia Boubendir}
\address{Department of Mathematics, Princeton University, Princeton, NJ 08544}
\email{lydia.boubendir@princeton.edu}
\author{Peter Constantin}
\address{Department of Mathematics, Princeton University, Princeton, NJ 08544}
\email{const@math.princeton.edu}
\author{Hezekiah Grayer II}
\address{Department of Mathematics, Stanford University, Stanford, CA 94305}
\email{hezekiah@stanford.edu}
\author{Zhongtian Hu}
\address{Department of Mathematics, University of Maryland, College Park, MD 20742}
\email{zh1077@umd.edu}

\usepackage[margin=1.5in]{geometry}
\usepackage{amsmath, amsthm, amssymb, bm}
\usepackage{times}
\usepackage{color}
\usepackage{microtype}
\usepackage{hyperref,comment}
\usepackage{graphicx}
\usepackage{subcaption}
\usepackage{tikz,siunitx,luatex85}
\usetikzlibrary{calc}
\usepackage[export]{adjustbox}

\newcommand{\blue}[1]{\textcolor{blue}{#1}}

\newcommand{\pa}{\partial}
\newcommand{\la}{\label}
\newcommand{\fr}{\frac}
\newcommand{\na}{\nabla}
\newcommand{\be}{\begin{equation}}
\newcommand{\ee}{\end{equation}}
\newcommand{\ba}{\begin{array}{l}}
\newcommand{\ea}{\end{array}}
\newcommand{\Rr}{{\mathbb R}}
\newtheorem{thm}{Theorem}
\newtheorem{prop}{Proposition}
\newtheorem{lemma}{Lemma}

\newtheorem{cor}{Corollary}
\newcommand{\beg}{\begin}

\newcommand{\eps}{\epsilon}
\newcommand{\R}{\mathbb{R}}

\newcommand{\bbS}{\mathbb{S}}
\newcommand{\U}{\mathbb{U}}
\newcommand{\T}{\mathbb{T}}
\renewcommand{\div}{{\mbox{div}\,}}
\newcommand{\D}{\Delta}
\newcommand{\V}{\mathbb{V}}

\newcommand{\p}{\partial}
\newcommand{\pt}{\p_t}

\newcommand{\pb}{\p_\beta}

\newcommand{\calI}{\mathcal{I}}

\newcommand{\calF}{\mathcal{F}}

\def\ygap{-25mm}

\DeclareMathOperator{\supp}{supp}

\date{today}

\begin{document}
\begin{abstract}
We consider two-fluid equations of inertial magnetohydrodynamics in two spatial dimensions.  These equations can be recast as the evolution of two active scalars,
with fluid-like interaction kernels. The equations come in two flavors, left-handed and right-handed. They are conservative and have global unique solutions for both  smooth initial data and for initial data in Yudovich class $L^1(\Rr^2)\cap L^{\infty}(\Rr^2)$. We prove that the inviscid limit holds in Yudovich class. We prove the persistence of 
$C^{2,\alpha}$ regularity of patch boundaries through merger events for the left-handed equations. We finally show numerical examples of magnetic reconnection, validating known rigorous results and uncovering new cases.
\end{abstract}

\maketitle
\section{Introduction}
Magnetic reconnection, a dynamical change of topology of the magnetic field lines, is a physical phenomenon of great importance to magnetohydrodynamics and plasma physics \cite{biskamp}, \cite{taylor1986}. The equations of ideal magnetohydrodynamics do not permit topology change in the absence of singularities. Magnetic resistivity can be responsible for a change of topology, but its time scale is in general too slow to explain the observed phenomena. A more rapid mechanism which does not require either singularities or resistivity was proposed in the physics literature \cite{boozer}, \cite{twoplasmas}, \cite{morrison}. The mechanism is due to inertia of electrons and it  is described using  fluid-like models with two fluids. The models are Hamiltonian, are  derived from action principles and have been studied numerically, but not by rigorous mathematics, as far as we know, before the work \cite{chu}. That work was devoted to two-dimensional theory and proved regularity results and topological change due to merger of two active scalars. 

The present paper is a continuation of \cite{chu}.  We consider two-dimensional inertial magnetohydrodynamics where evolution of the electronic vorticity $\omega$ is coupled to the evolution of the magnetic canonical momentum $F$. The system of equations is
\be
\left\{
\begin{aligned}
\pa_t F + \{\psi, F\} + \{\phi, \omega\} &= 0,\\
\pa_t \omega + \{\psi, \omega\} + \{\phi, F\} &=0,
\end{aligned}
\right .
\la{origeq}
\ee
in $\Rr^2$, where $\{f, g\} = \pa_1f\pa_2 g-\pa_2f\pa_1g$ is the Poisson bracket. 
 The electron fluid's stream function $\psi$ is related to the vorticity by 
\begin{equation}
    \Delta \psi = \omega,
    \la{psiomega}
\end{equation}
and the magnetic potential $\phi$ is related to the magnetic canonical momentum by
\begin{equation}
   -\Delta \phi + \phi  = \pm F. 
\la{phiF}
\end{equation}
The different $\pm$ cases are distinct, but are both conservative and derived from action principles. In \cite{chu} and in the present work, the  $(+)$ case is called \emph{right-handed}, and the $(-)$ case is called \emph{left-handed}. Although only the right-handed system was considered in the physics literature 
\cite{twoplasmas}, \cite{morrison}, both cases are mathematically interesting, yet quite different. Both systems can be re-written by a linear change of the dependent variables as evolution systems of two active scalars. Active scalars are real valued functions of space and time which evolve in time by pure transport, carried  by incompressible velocities they instantaneously  create. The velocities are computed from the scalars using  time-independent, constant coefficient, and  (in the present case) linear
equations of state. The Yudovich class is the class of active scalars whereby (both) the scalars are bounded and integrable in space. Patches are particular cases, where the scalars are indicator functions of bounded domains in $\Rr^2$.  The Yudovich class is preserved in time, and weak solutions exist and are unique in it 
\cite{chu}. Each of the two flavors (right-handed and left-handed) has screened and unscreened versions. In the unscreened versions \eqref{phiF} is replaced by 
$-\Delta \phi = \pm F$, i.e. the term $\phi$ is dropped. (That term physically corresponds to screening by an electronic skin depth).

The key feature of interest in the model \eqref{origeq} is the prediction of magnetic reconnection, i.e., topological change of the magnetic field $B = \nabla^{\perp} \phi$, without resistivity. Until recently, this feature was only captured  in numerical simulations  \cite{twoplasmas}.  In \cite{chu} the global well-posedness of \eqref{origeq} was obtained both in the Yudovich class and in the smooth category. Topological change was shown to occur in both the case of smooth solutions and patch solutions, and the results were for both the right-handed and left-handed systems. 

In this paper we prove the convergence of solutions of vanishing resistivity to the solutions of the inertial MHD system \eqref{origeq}. This is done by adapting the method of \cite{cde} to the case of two fluids. The result of \cite{cde} was for the two dimensional torus $\mathbb T^2$, and extending it to the whole space
presents certain challenges. One of the key elements of the result of \cite{cde} was the exponential integrability of the gradient of velocity. In Yudovich class the velocity is not Lipschitz-continuous, and the exponential integrability is a weaker surrogate for that. In the whole space, obviously there cannot be a global exponential integrability, and we have to localize the result. In addition, the decay of velocities is too slow for energy methods, and the velocities do not have self-contained equations like in the Euler and Navier-Stokes equations. We use methods form \cite{cwu} to control decay at spatial infinity of the active scalars, in the resistive case, and overcome thus the technical obstructions to global Yudovich class zero resistivity  limit. The proof is presented in Section \ref{sec:Yudo}.

The left-handed unscreened system consists of a pair of incompressible Euler equations which do not interact with each other. As it is known \cite{bc}, \cite{chemin}, in the Euler equations vortex patch boundaries stay smooth through the evolution, if they start smooth. In Section  \ref{sec:patc}
we show, for the left-handed screened system, that  $C^{1,\alpha}$- and $C^{2,\alpha}$-regularity of patch boundaries persists globally, even through merger. 

We conduct a numerical study of reconnection for patch solutions. In Section \ref{sec:nume1} we describe the algorithm, and  in Section \ref{sec:nume2} we describe the results of the numerical simulations.  These are as follows. 
For the left-handed screened equations,  initialized by a disk and an ellipse, we show  merger and then un-merger as the ellipse rotates through the disk,
in complete analogy with the pure Eulerian dynamics of the unscreened equations.
 In the right-handed unscreened case,  initialized by two patch disks with the same sign,  we observe numerically translational merger. The patches travel along a common path and merge. When the right-handed unscreened equations  are initialized by two patch disks of opposite signs, we observe orbiting patches deforming but not merging, indicating the possibility of  (quasi) periodic solutions. We also show that the symmetry that was proved rigorously in \cite{chu} to lead to merger, does so  numerically as well. Namely, we show that for right-handed unscreened and screened equations, initialized by four disks, with the same sign across the horizontal axis, and opposite sign across the vertical axis, there is merger across the horizontal axis.

\section{The active scalar system}
The sum and difference of the equations in the system \eqref{origeq} yield
\be
\left\{
\begin{aligned}
\pa_t (F + \omega) + \{\psi + \phi , F + \omega \} &= 0,\\
\pa_t (F - \omega)+ \{\psi - \phi , F - \omega\} &=0.
\end{aligned}
\right .
\ee
We define the scalars
 \be
 \sigma_{\pm} = F\pm \omega
 \la{sigmapm}
 \ee
 and organize them into the vector 
 \be
 \sigma = (\sigma_+, \sigma_-).
 \la{sigma}
 \ee
 The inertial MHD system \eqref{origeq} is  thus equivalent to the transport equations 
 \be
 \pa_t \sigma_{\pm} + \mathbb V_{\pm}(\sigma)\cdot\na \sigma_{\pm} =0,
 \la{sigmaeq}
 \ee
 where in the \emph{right-handed} system 
  \be
 \mathbb V_{\pm}(\sigma) = \mathbb U(\omega)\pm \mathbb B(F),
 \la{mathbbV}
 \ee
 and in the \emph{left-handed} system
   \be
 \mathbb V_{\pm}(\sigma) = \mathbb U(\omega)\mp \mathbb B(F).
 \la{mathbbV2}
 \ee
 The operators are 
 \be
 \mathbb U (\omega) = \na^{\perp}(\D^{-1})\,\omega,
 \la{mathbbU}
 \ee
 and
 \be
 \mathbb B( F) = \na^{\perp}(1 - \D)^{-1}F.
 \la{mathbbB}
 \ee
By construction,  $\mathbb U, \mathbb B$ and $ \mathbb V_{\pm}$ have divergence free ranges. The operator $\mathbb U$ is the Biot--Savart law from the incompressible Euler equations, and the operator $\mathbb{B}$ is a screened version of $-\mathbb{U}$.


\section{The ideal  inertial MHD limit}
\label{sec:Yudo}
We first introduce the viscous version of \eqref{sigmaeq}:
\begin{equation}
    \label{sigmaeqvisc}
    \pa_t \sigma_{\pm}^{\nu} + \mathbb V_{\pm}(\sigma^{\nu})\cdot\na \sigma_{\pm}^{\nu} - \nu \Delta \sigma_\pm^{\nu} =0,\quad \nu > 0.
\end{equation}
This is the setting where the conductive fluid has  momentum dissipation and ohmic dissipation of equal rate  $\nu$. The limit $\nu \to 0$ recovers ideal inertial MHD \eqref{origeq}, or equivalently, \eqref{sigmaeq}. We show that any Yudovich solution of compact support to \eqref{sigmaeq} may be recovered in this limit in a strong sense. We consider equal viscosities only for simplicity of exposition. The results hold for viscosities that vanish independently. We denote the space $\mathbb Y = L^1(\R^2) \cap L^\infty (\R^2)$ and its product $\mathbb Y^2 = \mathbb Y \times \mathbb Y$.

We recall here information about the solutions.
\beg{prop}\la{prel}
We consider initial data $\tau = (\tau_+, \tau_-) \in \mathbb Y^2 $ with compact support included in a ball of radius $L$ and we take an arbitrary  time interval $[0,T]$.
We assume that 
\be
\|\tau\|_{L^{\infty}} \le \Omega_{\infty}.
\la{liftyib}
\ee
The unique Yudovich solution of system \eqref{sigmaeq} and the unique solution of \eqref{sigmaeqvisc} with the same initial data obey
\be
\|\mathbb V_{\pm}(\sigma)\|_{L^{\infty}} \le V_{\infty} 
 \la{liftyb}
 \ee
 where
 \be
 V_{\infty} =  CL\Omega_\infty
 \la{vifty}
 \ee
with $C>0$ an absolute constant. We consider the ball
\be
B_T = \{x\in\Rr^2\left | \right. |x|\le L\left(1 +V_{\infty}\fr{T}{L}\right)\}.
\la{BT}
\ee
Then, the support of the solution $\sigma(t)$ of \eqref{sigmaeq} is included in $B_T$. The solution of \eqref{sigmaeqvisc} with the same initial data obeys
\be
\| \sigma^{\nu}(t)\|_{L^p(\Rr^2\setminus D_{T})} \le C\|\tau\|_{L^p}e^{-\frac{LV_\infty}{\nu}}
\la{outbound}
\ee
for any $1\le p\le \infty$ and all $t\le T$, where 
\be
D_T = \{x\in\Rr^2\left | \right. |x|\le 9L\left (2 +V_{\infty}\fr{T}{L}\right)\}.
\la{DT}
\ee
We took $Re = \fr{V_{\infty}L}{\nu}\ge 1$. 
\end{prop}
\beg{proof}
The bound \eqref{liftyb} follows from the Biot-Savart law and from the fact that $L^p$ norms of $\sigma^{\nu}$ are non-increasing in time. The proof of \eqref{outbound} follows the proof of Proposition 4 of \cite{cwu}: the viscous solutions $\sigma^{\nu}$ obey bounds with exponential weights $e^{\frac{V_{\infty}}{\nu}\sqrt{1 + |x|^2}}$. The expanding domain is $\{x\in\Rr^2\left | \right. |x|\le 9L(1 +Re^{-1}+V_{\infty}\fr{T}{L})\}\subset D_T$.
\end{proof}
\begin{thm}
    \label{thm:main}
    Let $\sigma = (\sigma_+, \sigma_-)$ be the unique Yudovich solution of system \eqref{sigmaeq} with initial data $\tau = (\tau_+, \tau_-) \in \mathbb Y^2 $ compactly  supported in a ball of radius $L$.   Let $T>0$ be arbitrary and let $\sigma^{\nu} = (\sigma_+^{\nu}, \sigma_-^{\nu})$ be the solution to system \eqref{sigmaeqvisc} with the initial data $\tau^{\nu} = (\tau_+^{\nu}, \tau_-^{\nu})$. Assume that the initial data belong to a ball in $\mathbb Y^2$, and in particular, 
  \be
  \|\tau\|_{L^{\infty}}\le \Omega_{\infty}, \|\tau^{\nu}\|_{L^{\infty}} \le \Omega_{\infty}.
  \la{initomegas}
  \ee
 Assume that $\tau$ is supported in the ball of radius $L$ and that \eqref{outbound} holds. For any $p \in [1,\infty)$, if 
 \be   
   \lim_{\nu\to 0}\|\tau^\nu -\tau\|_{L^p} = 0,
   \la{initconv}
   \ee
   then
    \be
    \lim_{\nu \to 0} \sup_{t \in [0,T]} \|\sigma^{\nu}_\pm - \sigma_\pm\|_{L^p} = 0.
    \la{concluth}
    \ee
\end{thm}
\beg{rem}
In view of Proposition \ref{prel}, the condition \eqref{outbound}  is verified if the viscous initial data are compactly supported in a ball of radius $L$.
\end{rem}

The proof follows the pattern of the proof in \cite{cde}, where the inviscid incompressible Euler limit was proved in the Yudovich class. Instead of scalar equations,  we are now considering  a system with two scalars. In addition, nontrivial adjustments need to be made because we are operating in the whole space, not on the torus. The velocity associated to a Yudovich $\sigma$, even if the latter is compactly supported, typically decays only like $\fr{1}{r}$, and thus fails to be in $L^2$. Also, a key element of \cite{cde} was  the global exponential integrability of gradients of velocity, which must be replaced here by local integrability. The strategy of proof is the same though. We prove a short time convergence on any time interval of order $c\Omega_{\infty}^{-1}$ if the initial data converge in $L^2$ and have the properties stemming from Proposition \ref{prel}. We obtain the long time limit by taking many steps (of the order $T\Omega_{\infty}$). If we start from the same initial data, the inviscid solution will be compactly supported in $B_T$ and the viscous solution will obey \eqref{outbound} and in particular it will be exponentially small outside the fixed larger set $D= D_{T}$, for all $t\le T$. We use these attributes  of the solutions throughout the proof. 
We prove the result for the right-handed system. The proof for the left handed system is the same.
\subsection{Proof for the right-handed system} In this section, we consider  the right-handed system given by \eqref{sigmaeq} and \eqref{mathbbV}, together with the viscous version 
\begin{equation}
    \label{eq:RSvisc}
    \begin{cases}
        (\pt  + \V_{+}(\sigma^{\nu}) \cdot \nabla  - \nu_+ \Delta) \sigma_+^{\nu_+} = 0,\\
        (\pt  + \V_{-}(\sigma^{\nu}) \cdot \nabla  - \nu_- \Delta) \sigma_-^{\nu_-} = 0.
    \end{cases}
\end{equation}
Above, $\sigma^{\nu} = (\sigma^{\nu_{+}}_{+}, \sigma^{\nu_{-}}_{-})$ for viscosities $\nu_{\pm} > 0$. The initial data is denoted $\tau^{\nu} = (\tau^{\nu_{+}}_{+}, \tau^{\nu_{-}}_{-})$. We  prove 
\begin{prop}\label{prop:mainRS}
    Let $\sigma, \sigma^{\nu}$ solve  right-handed systems \eqref{sigmaeq} and \eqref{eq:RSvisc}, respectively, with initial data $\tau, \tau^\nu \in \mathbb Y^2 $. Suppose that $\tau$ is compactly supported in the ball $B_T$ , assume both $\tau$ and $\tau^\nu$ obey \eqref{liftyib}   and that $\sigma^{\nu}$ obeys \eqref{outbound}.
    Then if
    $$
    \lim_{\nu_\pm \to 0} \|\tau_\pm^{\nu_\pm} - \tau_\pm\|_{L^2} = 0,
    $$
    there exists a universal constant $C_* > 0$ such that 
    $$
    \lim_{\nu_\pm \to 0} \sup_{t \in [0,T_*]} \|\sigma^{\nu_\pm}_\pm - \sigma_\pm\|_{L^2} = 0,
    $$
    where $T_* = (C_* \Omega_\infty)^{-1}$.
\end{prop}
 Proposition \ref{prop:mainRS} implies a more general version of Theorem \ref{thm:main} where $\nu$ is replaced by  possibly distinct $\nu_{\pm}$.  To prove Proposition \ref{prop:mainRS}, we consider the transport of mollified initial data.  We fix a scale $l > 0$ and consider $\sigma_l = (\sigma_{+,l}, \sigma_{-,l})$ which satisfies the system

\begin{equation}
    \label{eq:RSl}
    \begin{cases}
        (\pt  + \V_+(\sigma) \cdot \nabla) \sigma_{+,l} = 0,\\
        (\pt  + \V_-(\sigma) \cdot \nabla) \sigma_{-,l} = 0,
    \end{cases}
\end{equation}
equipped with initial data $\tau_l := \varphi_l * \tau$, where $\varphi_l$ is the standard mollifier at scale $l$; moreover, consider $\sigma_l^\nu = (\sigma_{+,l}^{\nu_+}, \sigma_{-,l}^{\nu_-})$ which satisfies the system
\begin{equation}
    \label{eq:RSviscl}
    \begin{cases}
        (\pt  + \V_+(\sigma^{\nu}) \cdot \nabla  - \nu_+ \Delta) \sigma_{+,l}^{\nu_+} = 0,\\
        (\pt  + \V_-(\sigma^{\nu}) \cdot \nabla  - \nu_- \Delta) \sigma_{-,l}^{\nu_-} = 0,
    \end{cases}
\end{equation}
equipped with data $\tau_l^\nu := \varphi_l * \tau^\nu$.  The transporting velocities are fixed; only the initial data have been mollified.


 We decompose
$$
\|\sigma_\pm^{\nu_\pm} - \sigma_\pm\|_{L^2} \le \|\sigma_\pm^{\nu_\pm} - \sigma_{\pm,l}^{\nu_\pm}\|_{L^2} + \|\sigma_{\pm,l}^{\nu_\pm} - \sigma_{\pm,l}\|_{L^2}  + \|\sigma_{\pm,l} - \sigma_\pm\|_{L^2}.
$$
Because the transporting velocities are the same, the first and the third term above converge to $0$ as we send $l$ to zero, uniformly in $ \nu_\pm$. It remains  to establish that
\begin{equation}\label{linconvergence}
\lim_{\nu_\pm \to 0} \sup_{t \in [0,T_*]} \|\sigma_{\pm,l}^{\nu_\pm} - \sigma_{\pm,l}\|_{L^2} = 0
\end{equation}
holds for a fixed $\ell$. We note that the inviscid initial data are smooth and compactly supported and that the viscous solution obeys \eqref{outbound} (by the same proof).
As a preliminary, we adapt two technical lemmas from \cite{cde}. The first one concerns the exponential integrability for the velocity:
\begin{lemma}
    \label{lem:expintRS}
    Let $\sigma \in \mathbb Y^2$ with
     \begin{equation*}
        \label{est:expintcond}
        \|\sigma\|_{L^{\infty}} \le \Omega_{\infty}.
    \end{equation*}
   Assume that  there exists a compact $D \subset \subset \R^2$ such that 
    \be
    \|\sigma\|_{L^q(\Rr^2\setminus D)} \le C\Omega_{\infty}
    |D|^{\fr{1}{q}} 
    \la{cond}
    \ee 
   holds for all $q\ge 1$. Then, there exist absolute non-dimensional constants $\gamma > 0$ and $C>0$  such that, for any measurable set $A$ we have
    \begin{equation}
        \label{est:expint}
        \int_{A} \exp\left(\fr{\gamma |\na \V_{\pm}(\sigma)(x)|}{\Omega_{\infty}} \right) dx \le C (|A| +|D|).
   \end{equation}
   \end{lemma}
\beg{rem} Note that the condition \eqref{cond} is satisfied by solutions of the inviscid equations with $D=B_T$ and solutions of viscous equations with $D= D_{T}$ in view of \eqref{outbound}.
\end{rem}
\beg{proof} The proof follows from the observation that the condition \eqref{cond} implies that $\|\sigma_{\pm}\|_{L^p} \le C\Omega_{\infty}|D|^{\fr{1}{p}}$ holds for $p\ge 1$. Then, expanding the exponential in power series and applying the Calderón--Zygmund $L^p$ bounds, and then  choosing  $\gamma$ small enough like in \cite{cde} concludes the proof. The zeroth order term introduces the dependence on $|A|$. The first order term needs to be interpolated between the zeroth order and the second order, because the CZ estimate fails there, but this changes only the prefactor in the right hand side of \eqref{est:expint}.   
\end{proof}
The second technical lemma is an adaptation of Lemma 2  in \cite{cde} for  $\R^2$. 
    \beg{lemma}\la{lem2r2}
    Let $u(x,t)$ be divergence-free in $\Rr^2$ and assume that $\omega =\na^{\perp}\cdot u\in \mathbb Y$ obeys $\|\omega\|_{L^{\infty}}\le \Omega_{\infty}$ and
    $\|\omega\|_{L^2} \le \Omega_2$. We also assume that there exists a set $D$ such that
    \be
    \|\omega\|_{L^q(\Rr^2\setminus D)} \le C\Omega_{\infty}
    |D|^{\fr{1}{q}}
    \la{condomega}
    \ee
    holds for all $q\ge1$. 
     Let  a non-negative $\theta$ solve 
    \be
    \pa_t\theta + u\cdot\na\theta -\nu\Delta\theta \le |\na u|\theta.
    \la{thetaineq}
    \ee
    Then for $T(p) = \fr{\gamma(p-1)}{2p\Omega_{\infty}} $ and any $p>1$, there exist constants $C_1$ and $C_2$ depending only on $\Omega_{\infty}$, $\Omega_2$ and $|D|$  such that
    \be
    \|\theta (t)\|_{L^2}^2 \le C_1\|\theta(0)\|_{L^{2p}}^{2p} +C_2
    \la{thetalp}
    \ee
    holds for $0\le t\le T(p)$. 
     \end{lemma}
    \beg{proof}
    We take a smooth, decreasing $p(t)\ge 1$ as in the proof of Lemma 2 in \cite{cde} and take the evolution of $\int \theta^{2p(t)}dx$. We have
\begin{equation}
    \begin{aligned}
    \fr{1}{2}\fr{d}{dt}\int \theta^{2p(t)}dx &= p'(t)\int (\log \theta )\theta^{2p(t)} dx  + p(t) \int \theta^{2p-1}\partial_t \theta \,dx\\
    &\le p'(t)\int (\log\theta ) \theta^{2p(t)} dx + p(t)\int\theta^{2p(t)}|\na u| \, dx.
    \la{pp}
    \end{aligned}
\end{equation}
    Now we take a large constant $M$ and consider the set $A = \{x \mid  |\na u|\ge M\}$. We have that $|A|\le \fr {\|\na u\|_{L^2}^2}{M^2}$. 
    On $A$ we use the inequality $ab \le e^a + b\ln b $ with $a=\gamma\fr {|\na u(x)|}{\Omega_{\infty}}$ and $b = \fr{\Omega_{\infty}}{\gamma}\theta^{2p(t)}$
    and Lemma \ref{lem:expintRS}. We obtain from \eqref{pp} and \eqref{est:expint}
   \begin{equation}
       \begin{aligned}
    \fr{1}{2}\fr{d}{dt}\int \theta^{2p(t)}dx\le &(p'(t) + \fr{2\Omega_\infty}{\gamma}p(t)^2 )\int (\log\theta) \theta^{2p(t)}dx\\
    &+ \left[p(t)\left(M + C(|D| + C\fr{\Omega_2^2}{M^2})\right) + \fr{\Omega_{\infty}}{\gamma}\log\left(\fr{\Omega_{\infty}}{\gamma}\right)\right]\int\theta^{2p(t)}\,dx.
        \end{aligned}
        \la{intpt}
   \end{equation}
    We choose $p(t)$ to solve $p'(t) + \fr{2\Omega_\infty}{\gamma}p(t)^2 =0$ and conclude as in the proof of Lemma 2 of \cite{cde}.
        \end{proof}
We apply this lemma to deduce an uniform in viscosity $H^1$ bound for the solutions of the transport of mollified initial data:

   \begin{lemma}\label{lem:H1RS}
    Let $l > 0$ be fixed. There exists a universal constant $C_* > 0$ and $C_l > 0$ depending only on $l$ and $\|\tau\|_{L^p}$, $p>2$ such that
    \be
    \sup_{t \in [0,T_*]} (\|\sigma_{\pm,l}^{\nu_\pm}\|_{H^1} + \|\sigma_{\pm,l}\|_{H^1}) \le  C_l,
    \la{h1lbound}
    \ee
    for $T_* = (C_* \Omega_\infty)^{-1}$.
\end{lemma}
\beg{proof}
     We show the proof of the bound for $\sigma_{+,l}^{\nu_+}$.    Differentiating the first equation of \eqref{eq:RSviscl}, we have
 \begin{equation}
    \pt \nabla \sigma_{+,l}^{\nu_{+}} - \V(\sigma^{\nu}) \cdot \nabla (\nabla \sigma_{+,l}^{\nu_{+}}) - \nu_+ \Delta \nabla \sigma_{+,l}^{\nu_{+}} = \nabla \V(\sigma^{\nu}) \cdot \nabla \sigma_{+,l}^{\nu_{+}},
    \end{equation}
    which implies
    \begin{equation}
    (\pt - \V(\sigma^{\nu}) \cdot \nabla - \nu_+\Delta)|\nabla \sigma_{+,l}^{\nu_{+}}| \le |\nabla \V(\sigma^{\nu})||\nabla\sigma_{+,l}^{\nu_{+}}|.
    \end{equation}
Then using Lemma \ref{lem2r2} with $\theta = |\nabla \sigma_{+,l}^{\nu_{+}}| $ and taking advantage of the fact that the initial data are bounded in $L^p$ with $p>2$ (paying in $\ell$) we conclude the proof.
\end{proof}
In order to prove \eqref{linconvergence}, we  note
\begin{equation}
\begin{aligned}
    &\frac12 \frac{d}{dt} \|\sigma_{\pm,l}^{\nu_\pm} - \sigma_{\pm,l}\|_{L^2}^2 \\
    &= - \int \left(\V_\pm(\sigma^\nu) - \V_\pm(\sigma)\right) \cdot \na \sigma_{\pm,l}^{\nu_\pm} (\sigma_{\pm,l}^{\nu_\pm} - \sigma_{\pm,l}) \, dx  -\nu_\pm \|\na \sigma_{\pm,l}^{\nu_\pm} - \na \sigma_{\pm,l}\|_{L^2}^2  \\
    &\quad - \nu_\pm \int  \nabla\sigma_{\pm,l} \cdot\nabla (\sigma_{\pm,l}^{\nu_\pm} - \sigma_{\pm,l}) \, dx \\
    &\le C\Omega_\infty C_l\|\V_\pm(\sigma^\nu) - \V_\pm(\sigma)\|_{L^\infty(0,T_*;L^2(D))} + \nu_\pm C_l^2 + C_\ell \|\tau\|_{\mathbb Y}^2e^{-\frac{LV_\infty}{\nu}}.
\end{aligned}
\end{equation}
Indeed, the portion outside $D$ of the integral involving $\V_\pm(\sigma^\nu) - \V_\pm(\sigma)$ is bounded by an $L^4-L^2-L^4$ H\"{o}lder bound, using the boundedness of the velocities in $L^4$, the bounds \eqref{h1lbound} and \eqref{outbound} with $p=4$.  We used Lemma \ref{lem:H1RS} for the last term in the equality.
In view of the differential inequality above, to establish the limit \eqref{linconvergence} and thus Proposition \ref{prop:mainRS}, it suffices to show that 
\begin{equation}
\label{eq:limVpm}
\lim_{\nu_\pm \to 0}\|\V_\pm(\sigma^\nu) - \V_\pm(\sigma)\|_{L^\infty(0,T_*;L^2(D))} = 0.
\end{equation}

 In other words, it suffices to prove an analogue of Lemma 4 in \cite{cde} for $\V_\pm(\sigma)$.  We denote for simplicity of exposition the difference as $ h  = \sigma_{\pm}^\nu -\sigma_{\pm}$ and the semisum as $H = \fr{1}{2} (\sigma_{\pm}^\nu + \sigma_{\pm})$. 
 Let $\psi, \Psi $ and $\phi, \Phi$  satisfy 
 \begin{equation}
 \Delta \psi = h, \qquad \Delta \Psi = H, \qquad  -\Delta \phi + \phi = h, \qquad -\Delta \Phi + \Phi = H.
 \end{equation}
In view of the definition of $\V_\pm$, we rewrite it as
\begin{equation}
    \label{Vpm2}
    \V_\pm(h) = \frac12\nabla^\perp(\psi  \pm \phi ).
\end{equation}
From \eqref{Vpm2}, we see that \eqref{eq:limVpm} amounts to controlling $\|\na \psi\|_{L^2(D)}$ and $\|\na \phi\|_{L^2(D)}$. We show how to control $\na \psi$, the bound of $\na\phi$ is similar.

 The equation for $h$ is
 \be
 \pa_t  h  + V\cdot\na h  + v\cdot\na  H  = \nu\Delta h  + \nu\Delta  H 
 \la{diffeq}
 \ee
 where  $v = \mathbb V_{\pm} (\sigma_{\pm}^\nu -\sigma_{\pm})  = \V_{\pm} h$ and $V=\fr{1}{2}\mathbb V_{\pm}(\sigma_{\pm} +\sigma_{\pm}^\nu) = \V_{\pm} H$. 
 Consider a smooth cutoff $\chi$ equal to 1 on a neighborhood of $D_{T}$ and compactly supported. More precisely, we consider $\chi_R :=\chi (\fr{|x|}{R})$ with $R$ sufficiently large and time independent, whose precise value will be made explicit below. For the ease of reading, we  denote $\chi_R$ by $\chi$.

We establish first some properties of $V, v,  h $ and $ H $.
For $ h $ and $ H$, we have Yudovich  bounds (time independent) and bounds which decay exponentially 
\be
\| e^{|\cdot |\fr{V_{\infty}}{\nu} t}[| h (\cdot,t)| + | H (\cdot,t)|]\|_{L^q}  \le C_q
\ee
for all $1\le q<\infty$. Then with the above bounds, we show that
\be
|\psi (x,t)| + |\Psi (x,t)| \le C \log R,
\la{psibound}
\ee
and
\be
|\na\psi (x,t)| + |\na\Psi (x,t)| \le C R^{-1}
\la{napsibound}
\ee
uniformly for $x\in \supp  \chi$.
Indeed, to prove \eqref{psibound} it suffices to obtain the following estimate:
\begin{equation}
\int \left|\log|x-y|\right| | H (y,t)|dy \le C\log R
\end{equation}
for $|x| \le R $. For $|x-y| \le 1$ we split $\log||x-y||\sqrt{ | H (y)|}\sqrt{| H  (y)|}$ and use $L^2-L^2$
to obtain $C\sqrt{\| H \|_{L^{\infty}}}$ from the first integral
 because $\log|z|$ is integrable on $|z|\le 1$,
and $\sqrt{ \|H \|_{L^1}}$ for the second integral. For the case $|x-y| \ge 1$ and $|y| \le R$, we use $L^1$ for $ H $ and that $|\log |x-y|| < \log R$ for $1 <|x-y| \le 2R$.
On the other hand for $|y|\ge R$ we use the fact that $|\log |y| H |$ is square integrable because of the rapid decay of $ H $.

To prove \eqref{napsibound}, it suffices to show that for $|x| \le R$,
\begin{equation}
    \left|\int \frac{x - y}{|x-y|^2} H(y,t) dy\right| \le C R^{-1}.
\end{equation}
If $|x-y| > \frac R2$, we bound $H$ by $L^1$ and the kernel by $CR^{-1}$. If $|x-y| \le \frac R2$, we estimate the kernel in $L^1$ and use the exponential decay of $H$.

We also need an estimate for $\Delta^{-1} \partial_t h$. In conservative form we have
\begin{equation}
   \partial_t h = \div ( -Vh - vH + \nu \nabla h + \nu \nabla H)
\end{equation} 
and use  bounds on Riesz potentials,
\[
\|R\Lambda^{-1} (u H  + U h )\|_{L^r} \le C\|u H  + U h \|_{L^q}
\]
with $\fr{1}{r} = \fr{1}{q} -\fr{1}{2}$ with $q<2$. We obtain that 
\be
\|\partial_t \psi\|_{L^r} \le  K^2  + \nu K
\la{psit}
\ee
with $r>2$ and  $K$ the Yudovich norm.

 We take \eqref{diffeq} and multiply by $- \chi \psi$. After  integrating by parts, we have
 \be
\begin{aligned}
 \fr{1}{2} \fr{d}{dt}\int& (\chi |\na\psi |^2) dx = \int (\na\psi \cdot \na V\na\psi +  H  v\cdot\na\psi)\chi dx \\
 &+ \int s(\nu)\cdot\na\chi dx
  - \nu\int ( h  +  H )\Delta(\chi\psi)dx + \int \partial_t\psi \nabla\cdot(\psi\nabla \chi)dx
  \end{aligned}
 \la{intonce}
 \ee
In the case of a single fluid $v\cdot\na\psi =0$, but for two fluids the cancellation does not hold. Nevertheless, the term has  a bounded multiplying factor $ H $. More care is required for the term with the vector $s(\nu)$
 \be
 s(\nu) = \psi( H  u -  h  U)  -\fr{1}{2} |\na\psi|^2 U  + (V\cdot\na\psi)\na\psi.
 \la{snu}
 \ee
 The first piece is bounded  in $L^1$ by $\log R$ times norms in $\|u H  - U h \|_{L^1}$.
 For the second and third piece, we bound $|\na \psi|^2 U$ in $L^\infty$ by $KR^{-2}$ after using \eqref{napsibound}, and bound $\na \chi$ in $L^1$ by $CR$.
 Therefore
 \be
 \left |\int s(\nu)\cdot\na\chi dx \right | \le K(1+\log R)R^{-1}
 \la{snusmall}
 \ee
 where $K$ depends on norms in $\mathbb Y$ of both $\sigma$ and $\sigma^{\nu}$ which are uniformly bounded and constant in time.
 The terms from the Laplacian when $\chi$ is  not differentiated are  $O(\nu)$.
 Then, 
 \begin{equation}
     \int \partial_t \psi\nabla \cdot (\psi\nabla \chi)dx = \int \partial_t  \psi \psi \Delta \chi \, dx + \int \partial_t \psi \na\psi \cdot  \na\chi \, dx.
 \end{equation}   
 gives two terms. We use \eqref{psibound} and 
$\partial_t\psi$ in $L^4$ with \eqref{psit} for the term $\partial_t  \psi \psi \Delta \chi$  and have  $\log R\|\Delta\chi\|_{L^{\fr{4}{3}}} = (\log R) R^{-\fr{1}{2}}$. In the term $\partial_t \psi \na\psi \cdot  \na\chi$ we use $L^4$ for $\na\chi$ and say $L^3$ for $\na\psi$ leaving $L^{\fr{12}{5}}$ for $\partial_t \psi$. 
 This term is thus bound by
 \be
\left  |\int \partial_t\psi \nabla\cdot (\psi\nabla \chi)dx\right | \le K (1+ \log R)R^{-\fr{1}{2}}
\la{psitterm}
\ee

Now we estimate the term with $\nabla \psi \cdot \nabla V \nabla \psi$. We split at height $M$ for the size of $|\na \psi|^2$.
 We get a contribution of the form 
 $KM^{-\fr{1}{2}}$ from the ``bad'' region:
 \begin{equation}
 \int_{\{ |\nabla \psi |^2 \geq M\}} \nabla \psi \cdot \nabla V \nabla \psi \chi dx \leq K M^{-\fr{1}{2}}
 \end{equation}
This power  comes from a priori bounds for $\na V\in L^4$ and for $\na\psi \in L^4$.  In the "good" region  we use the inequality
 \[
 ab \le \epsilon e^a + \log{\epsilon^{-1}}b + b\log b
 \]
 which follows from the familiar $ab\le e^a + b\log b -b$ for nonnegative $a,b$ by considering whether $a\ge \log\epsilon^{-1}$ or not.
 If $a\le \log\epsilon^{-1}$ then we obtain $ab \le (\log\epsilon^{-1})b$. If it is greater, we use the inequality $a'b\le e^{a'} +b\log b$ for $a' = a + \log \epsilon$.
 We obtain an error $\epsilon K$. Therefore
  \begin{equation}
 \int_{\{ |\nabla \psi |^2 < M\}} \nabla \psi \cdot \nabla V \nabla \psi \chi dx \leq \epsilon K + \log \left(\fr{M^2}{\epsilon}\right ) \int\chi |\na\psi|^2 dx 
 \end{equation}
 
 Overall, the right hand side of \eqref{intonce} is bounded by
 \be
  H _{\infty}\left[1+\log \left(\fr{M^2}{\epsilon}\right )\right]\int\chi |\na\psi|^2 dx + (\nu + y_0)K_0
 \la{RHS}
 \ee
 where we set
 \be
(\nu + y_0)K_0 = \epsilon K + KM^{-\fr{1}{2}} + KR^{-\fr{1}{2}}\log R + K\nu.
 \la{kzeronu}
 \ee 
 with $y_0 = \int \chi |\na\psi(0)|^2dx$. This is obtained by taking independently  and explicitly $\epsilon$  and $M^{-1}$  in terms of $\nu + y_0$ and $R$ in terms only of $\nu$.  (We could have done this for $M^{-1}$ and $\epsilon$ as well, but we need (and afford) them to be slightly larger.)
 We write
 \be 
 y(t) = \int\chi |\na\psi|^2 dx.
 \la{y}
 \ee
 The diffential inequality is
 \be
 \fr{dy}{dt} \le \left [  H _{\infty}\left(1  +\log \left(\fr{M^2}{\epsilon}\right )\right)\right]y(t) + K_0(\nu + y_0)
 \la{ode}
 \ee
 Our integrating factor is $e^{- H _{\infty}(1+\log (\fr{M^2}{\epsilon}) )t} = e^{- H _{\infty} t}\left (\fr{M^2}{\epsilon}\right)^{-t H _{\infty}}$. We multiply by it and integrate
 \be
\begin{aligned}
 y(t) &\le \left [y(0)  + \fr{K_0(\nu + y_0)}{  H _{\infty}(1 +\log (\fr{M^2}{\epsilon}))}\right ]e^{ H _\infty t} \left (\fr{M^2}{\epsilon}\right)^{t H _{\infty}}\\
 &\le ( H _{\infty}+K_0) H _{\infty}^{-1} (\nu + y_0)e^{ H _\infty t} \left (\fr{M^2}{\epsilon}\right)^{t H _{\infty}}
 \end{aligned}
  \la{notyet}
 \ee
  Now $\fr{M^2}{\epsilon} = K(\nu + y_0)^{-A}$ with a fixed $A>0$ and for $At H _{\infty}<1$ we have convergence to zero as $\nu +y_0\to 0$. 
  
  We have shown the required convergence for $\na \psi$  in the inviscid limit, and a similar argument yields convergence for $\na \phi$. In particular, \eqref{linconvergence} and thus Proposition \ref{prop:mainRS} is proved. Theorem \ref{thm:main} follows via iteration, see \cite{cwu}.

\section{Patch regularity for left-handed equations}
\label{sec:patc}
In this section, we prove the propagation of regularity for the patch boundary in H\"older spaces. We show that the $C^{2,\alpha}$ boundary regularity can  be preserved even when a $\sigma_+$ patch intersects transversally with a $\sigma_-$ patch. 
\subsection{Propagation of $C^{1,\alpha}$ regularity}
Fix $\alpha \in (0,1)$. We start with the propagation of $C^{1,\alpha}$ regularity for patch solutions to the Left-handed, Screened (LS) model by following the approach \textit{\`a la} \cite{bc}. We recall from \cite{chu} that the (LS) model can be equivalently formulated as
\begin{equation}
    \label{eq:LS}
    \pt \sigma_\pm + v_\pm \cdot \na \sigma_\pm = 0,\quad v_\pm = \pm \U(\sigma_\pm) \mp \bbS(F),
\end{equation}
where $\bbS := \na^\perp (\mathbb I -\Delta)^{-1}\Delta^{-1}$, $F = \frac{\sigma_+ + \sigma_-}{2}$.

Following \cite{bc}, we reformulate the patch problem as follows: let $\sigma_\pm = a_\pm\chi_{\Omega_\pm},$ where $\Omega_\pm$ are bounded domains in $\R^2$. We consider scalars $\varphi_\pm$ solving
\begin{equation}
    \label{eq:phipm}
    \begin{cases}
    \pt \varphi_\pm + v_\pm \cdot \na \varphi_\pm = 0,\\
    \varphi_\pm(x,0) = \varphi_{\pm,0}(x),
    \end{cases}
\end{equation}
where $v_\pm$ can be explicitly written as:
\begin{equation}
    \label{eq:vpm}
    \begin{aligned}
    v_\pm(x) &= \pm \frac{a_\pm}{2\pi} \int_{\Omega_\pm(t)} \na_x^\perp \log|x-y| dy\\
    &\quad \mp \left(\frac{a_+}{4\pi} \int_{\Omega_+(t)} \nabla^\perp_x G(|x-y|) dy + \frac{a_-}{4\pi} \int_{\Omega_-(t)} \nabla^\perp_x G(|x-y|) dy\right) \\
    &=: \pm \U(\sigma_\pm) \mp \bbS(F)
    \end{aligned}
\end{equation}
where for $r > 0$, $G(r) = K_0(r) + \log r$ and $K_0$ is the modified Bessel function. Moreover, the patches $\Omega_\pm$ are given by
\begin{equation}\label{defn:Omegapm}
\Omega_\pm(t) = \{\varphi_\pm(x,t) > 0\}.
\end{equation}
We assume that the initial domains $\Omega_{\pm, 0}$ are defined as:
\begin{equation}\label{defn:Omegapm0}
\Omega_{\pm, 0} = \{\varphi_{\pm,0}(x) > 0\},
\end{equation}
where $\varphi_{\pm,0} \in C^{1,\alpha}(\R^2)$, and the following nondegeneracy condition holds:
\begin{equation}
    \label{nondegeneracy}
    \inf_{x \in \p\Omega_{\pm,0}} |\na \varphi_{\pm,0}(x)| \ge m > 0.
\end{equation}

We also introduce the following notations: since $\sigma_\pm$ is transported by incompressible flows $v_\pm$, we set
$$
(L_\pm)^2 := |\Omega_\pm(t)| = |\Omega_{\pm, 0}|.
$$
With the above notation, we record the following conserved quantities:
\begin{equation}
    \label{eq:patchLp}
    \|\sigma_\pm\|_{L^p} = a_\pm (L_\pm)^{\frac2p},\quad p\in [1,\infty].
\end{equation}
We also introduce the notations: given a function $f$,
\begin{align*}
    |\na f|_{\inf} := \inf_{x \in \p \Omega_\pm(t)} |\na f(x,t)|,\\
    |f|_\alpha := \sup_{x \neq x'}\frac{|f(x) - f(x')|}{|x-x'|^\alpha}.
\end{align*}

We first show the following result concerning the propagation of the $C^{1,\alpha}$ regularity of $\varphi_\pm$.

\begin{thm}\label{thm:C1alpha}
    Suppose $a_\pm \neq 0$, $\Omega_{\pm,0}$, $\varphi_{\pm,0} \in C^{1,\alpha}$, $\alpha \in (0,1)$, such that \eqref{defn:Omegapm0} and \eqref{nondegeneracy} hold. There exists a constant $C > 0$ depending only on initial data and $\alpha$ that the problem \eqref{eq:phipm}--\eqref{defn:Omegapm0} has a unique solution $\varphi_\pm(x,t) \in C^{1,\alpha}(\R^2)$ with the following bounds:
    \begin{align}
        \|\na v_\pm(\cdot,t)\|_{L^\infty} &\le C\|\na v_\pm(\cdot,0)\|_{L^\infty} e^{Ct},\label{est:navpminfty}\\
        |\na \varphi_\pm(\cdot,t)|_{\alpha} &\le C(|\na \varphi_\pm(\cdot,0)|_{\alpha} + \|\na \varphi_\pm(\cdot,0)\|_{L^\infty}t) e^{Ct + C_0e^{Ct}},\\
        \|\na \varphi_\pm(\cdot,t)\|_{L^\infty} &\le \|\na \varphi_\pm(\cdot,0)\|_{L^\infty} e^{e^{Ct}},\label{est:naphipminfty}\\
        |\na \varphi_\pm(\cdot,t)|_{\inf} &\ge |\na \varphi_\pm(\cdot,0)|_{\inf} e^{-e^{Ct}},
    \end{align}
    for $t \ge 0$, where $C_0$ is a dimensional constant.
\end{thm}

The rest of this subsection is dedicated to the proof of Theorem \ref{thm:C1alpha}. We arrange the proof into the following steps.

\noindent\textbf{Step 1: Estimates for the singular part $\U(\sigma_\pm)$.} Since the operator $\U$ coincides with the classical Biot-Savart operator for 2D incompressible Euler, we record the following results from \cite{bc} concerning Euler patches:
\begin{prop}\label{prop:naU}
    Let $\varphi_\pm$ be related to $\Omega_\pm$ as in \eqref{defn:Omegapm} and \eqref{defn:Omegapm0}. Then for $\alpha \in (0,1)$, the following inequality holds:
    \begin{equation}
        \label{est:LinftynaU}
        \|\na \U(\sigma_\pm)\|_{L^\infty} \le \left(\frac92 + \frac2\alpha\right)|a_\pm| \left(1 + \log \Delta^\pm_\alpha\right),
    \end{equation}
    where
    $$
    \Delta_\alpha^\pm := \frac{|\na\varphi_\pm|_\alpha L_\pm^\alpha}{|\na\varphi_\pm|_{\inf}}.
    $$
\end{prop}
Moreover, we have the following result:
\begin{prop}
    There exists a universal constant $C_0 > 0$ such that if $W_\pm \in C^\alpha (\R^2,\R^2)$ is a divergence-free vector field tangential to $\p \Omega_\pm$, then
    \begin{equation}
        \label{est:HoldernaUW}
        |\na\U(\sigma_\pm) W_\pm|_\alpha \le C_0\|\na \U(\sigma_\pm)\|_{L^\infty} |W_\pm|_\alpha.
    \end{equation}
\end{prop}

\noindent\textbf{Step 2: Estimates for the regular part $\bbS(F)$.} In this step, we first prove the following proposition concerning the regularity of $\na \bbS(F)$:
\begin{prop}
    Let $\sigma_\pm = a_\pm \chi_{\Omega_\pm}$ and $F = \frac12 (\sigma_+ + \sigma_-)$. The following bound holds:
    \begin{equation}
        \label{est:SF1}
        \|\na \bbS(F)\|_{H^2} \le \|F\|_{L^2}.
    \end{equation}
    In particular, the following bound also holds: for $\alpha \in (0,1)$,
    \begin{equation}
        \label{est:SF2}
        \|\na\bbS(F)\|_{L^\infty} + |\na \bbS(F)|_\alpha \le C_\alpha (|a_+| L_+ + |a_-| L_-).
    \end{equation}
\end{prop}
\begin{proof}
    To prove the proposition, it is more convenient to use the Fourier representation:
    $$
    \Hat{\na \bbS(F)}(\xi) = \frac{\xi^\perp \xi}{|\xi|^2 (1+ |\xi|^2)}\Hat{F}(\xi).
    $$
    Then by Plancherel's identity and the above identity, we see that
    $$
    \|\na \bbS(F)\|_{H^2}^2 = \int_{\R^2} (1+|\xi|^2)^2 \left|\frac{\xi^\perp \xi}{|\xi|^2 (1+ |\xi|^2)}\Hat{F}(\xi)\right|^2 d\xi \le \|F\|_{L^2}^2,
    $$
    which proves \eqref{est:SF1}. The estimate \eqref{est:SF2} follows from Sobolev embedding and \eqref{eq:patchLp}.
\end{proof}

The following corollary then holds:
\begin{cor}
    If $W \in (L^\infty \cap C^\alpha) (\R^2,\R^2)$ is an arbitrary divergence-free vector field, then
    \begin{equation}
        \label{est:HoldernaSW}
        |\na\bbS(F) W|_\alpha \le C_\alpha (|a_+| L_+ + |a_-| L_-)(\|W\|_{L^\infty} + |W|_\alpha),
    \end{equation}
    where $C_\alpha$ is defined in \eqref{est:SF2}.
\end{cor}

\noindent\textbf{Step 3: Time-dependent estimate for $\na^\perp \varphi_\pm$.} 
We prove the following result:
\begin{prop}\label{prop:evolution}
    Let $T > 0$. Assume that $\varphi_\pm$ is a solution to \eqref{eq:phipm} on time interval $[0,T)$. Suppose that $\varphi_\pm \in C^{1,\alpha}$ and $|\na\varphi_\pm|_{\inf} > 0$ for all $t \in [0,T)$. Then
    \begin{equation}
        \label{est:Calphanaphi}
        \begin{aligned}
        |\na\varphi_\pm&(\cdot,t)|_\alpha \le \left[|\na \varphi_\pm(\cdot,0)|_\alpha + \bar{C}(a_\pm, L_\pm, \alpha)\|\na \varphi_\pm(\cdot,0)\|_{L^\infty}t\right]\\
        &\times \exp \left(2\bar{C}(a_\pm, L_\pm, \alpha)t + (C_0+1) \int_0^t \| \na\U_\pm(\sigma_\pm)(\cdot,s)\|_{L^\infty}ds\right),
        \end{aligned}
    \end{equation}
    \begin{equation}
        \label{est:Linftynaphi}
        \|\na\varphi_\pm(\cdot,t)\|_{L^\infty} \le \|\na\varphi_{\pm,0}\|_{L^\infty}\exp\left(\int_0^t \|\na v_\pm(\cdot,s)\|_{L^\infty} ds\right),
    \end{equation}
    \begin{equation}
        \label{est:infnaphi}
        |\na\varphi_\pm(\cdot,t)|_{\inf} \ge \|\na\varphi_{\pm,0}\|_{\inf}\exp\left(-\int_0^t \|\na v_\pm(\cdot,s)\|_{L^\infty} ds\right).
    \end{equation}
\end{prop}
\begin{proof}
    From \eqref{eq:phipm}, we see that $W_\pm :=\na^\perp \varphi_\pm$ satisfies
    \begin{equation}
        \label{eq:Wpm}
        \pt W_\pm + v_\pm \cdot \na W_\pm = \na v_\pm W_\pm.
    \end{equation}
    Then estimates \eqref{est:Linftynaphi} and \eqref{est:infnaphi} follow on the Lagrangian paths of $v_\pm$. We then focus on the proof of \eqref{est:Calphanaphi}. By following an identical computation to that in \cite{bc}, the following estimate holds:
    \begin{equation}
        \label{est:Calphanaphiaux1}
        \begin{aligned}
        |W_\pm(\cdot,t)|_\alpha &\le |W_\pm(\cdot,0)|_\alpha \exp\left(\alpha\int_0^t \|\na v_\pm(\cdot,s)\|_{L^\infty} ds\right)\\
        &\quad+ \int_0^t |\na v_\pm W_\pm|_\alpha \exp\left(\alpha \int_s^t \|\na v_\pm(\cdot,s')\|_{L^\infty} ds'\right)ds.
        \end{aligned}
    \end{equation}
    Denoting $Q_\pm(s) = \|\na v_\pm(\cdot,s)\|_{L^\infty}$ and using estimates \eqref{est:HoldernaUW}, \eqref{est:HoldernaSW}, we further have
    \begin{align*}
        |W_\pm(\cdot,t)|_\alpha &\le |W_\pm(\cdot,0)|_\alpha \exp\left(\alpha\int_0^t Q_\pm(s) ds\right)\\
        &\quad+ \int_0^t (\bar{C}(a_\pm, L_\pm, \alpha) + C_0 \|\na \U(\sigma_\pm)\|_{L^\infty})|W_\pm(\cdot,s)|_\alpha \exp\left(\alpha \int_s^t Q_\pm(s') ds'\right) ds\\
        &\quad + \int_0^t (\bar{C}(a_\pm, L_\pm, \alpha)|W_\pm(\cdot,s)|_{L^\infty}) \exp\left(\alpha \int_s^t Q_\pm(s') ds'\right) ds,
    \end{align*}
    where $\bar{C}(a_\pm, L_\pm, \alpha) := C_\alpha (|a_+| L_+ + |a_-| L_-)$.

    Multiplying on both sides by $\exp\left(-\alpha \int_0^t Q_\pm(s) ds\right)$ and defining 
    $$G_\pm(t) := |W_\pm(\cdot,t)|_\alpha\exp\left(-\alpha \int_0^t Q_\pm(s) ds\right),$$ 
    we deduce that 
    \begin{align*}
        G_\pm(t) &\le |W_\pm(\cdot,0)|_\alpha + \int_0^t (\bar{C}(a_\pm, L_\pm, \alpha) + C_0 \|\na \U(\sigma_\pm)\|_{L^\infty}) G_\pm(s) ds\\
        &\quad + \int_0^t (\bar{C}(a_\pm, L_\pm, \alpha)|W_\pm(\cdot,s)|_{L^\infty}) \exp\left(-\alpha \int_0^s Q_\pm(s') ds'\right) ds \\
        &\le |W_\pm(\cdot,0)|_\alpha + \int_0^t (\bar{C}(a_\pm, L_\pm, \alpha) + C_0 \|\na \U(\sigma_\pm)\|_{L^\infty}) G_\pm(s) ds\\
        &\quad + \bar{C}(a_\pm, L_\pm, \alpha)\|W_\pm(\cdot,0)\|_{L^\infty}\int_0^t  \exp\left((1-\alpha) \int_0^s Q_\pm(s') ds'\right) ds,
    \end{align*}
    where we used \eqref{est:Linftynaphi} in the last inequality above. Now applying Gr\"onwall inequality, we arrive at
    \begin{align*}
        G_\pm(t) &\le\left[|W_\pm(\cdot,0)|_\alpha + \bar{C}(a_\pm, L_\pm, \alpha)\|W_\pm(\cdot,0)\|_{L^\infty}\int_0^t  \exp\left((1-\alpha) \int_0^s Q_\pm(s') ds'\right) ds\right]\\
        &\quad \times \exp \left(\bar{C}(a_\pm, L_\pm, \alpha)t + C_0 \int_0^t \|\na \U(\sigma_\pm)(\cdot,s)\|_{L^\infty}ds\right),
    \end{align*}
    Multiplying on both sides by $\exp\left(\alpha \int_0^t Q_\pm(s) ds\right)$, we obtain the following estimate:
    \begin{align*}
        &|\na\varphi_\pm(\cdot,t)|_\alpha \le\\
        & \left[|W_\pm(\cdot,0)|_\alpha + \bar{C}(a_\pm, L_\pm, \alpha)\|W_\pm(\cdot,0)\|_{L^\infty}\int_0^t  \exp\left((1-\alpha) \int_0^s Q_\pm(s') ds'\right) ds\right]\\
        &\quad\times \exp \left(\bar{C}(a_\pm, L_\pm, \alpha)t + C_0 \int_0^t \|\na \U(\sigma_\pm)(\cdot,s)\|_{L^\infty}ds +\alpha \int_0^t\|\na v_\pm(\cdot,s)\|_{L^\infty} ds\right).
    \end{align*}
    Recall that $\na v_\pm = \pm\na \U(\sigma_\pm) \mp \na \bbS(F)$ and the bound \eqref{est:SF2}, we have
    $$
    \|\na v_\pm\|_{L^\infty} \le \|\na\U(\sigma_\pm)\|_{L^\infty} + \bar{C}(a_\pm, L_\pm, \alpha).
    $$
    Moreover, we have the bound
    $$
    \int_0^t  \exp\left((1-\alpha) \int_0^s Q_\pm(s') ds'\right) ds \le t\exp\left((1-\alpha) \int_0^t \|\na v_\pm(\cdot,s)\|_{L^\infty} ds\right).
    $$
    Combining the above three estimates, we finally arrive at
    \begin{align*}
        &|\na\varphi_\pm(\cdot,t)|_\alpha \le \left[|W_\pm(\cdot,0)|_\alpha + \bar{C}(a_\pm, L_\pm, \alpha)\|W_\pm(\cdot,0)\|_{L^\infty}t\right]\\
        &\quad\times \exp \left(\bar{C}(a_\pm, L_\pm, \alpha)t + C_0 \int_0^t \|\na \U(\sigma_\pm)(\cdot,s)\|_{L^\infty}ds +\int_0^t\|\na v_\pm(\cdot,s)\|_{L^\infty} ds\right)\\
        &\le \left[|W_\pm(\cdot,0)|_\alpha + \bar{C}(a_\pm, L_\pm, \alpha)\|W_\pm(\cdot,0)\|_{L^\infty}t\right]\\
        &\quad\times \exp \left(2\bar{C}(a_\pm, L_\pm, \alpha)t + (C_0+1) \int_0^t \|\na \U(\sigma_\pm)(\cdot,s)\|_{L^\infty}ds\right),
    \end{align*}
    which yields \eqref{est:Calphanaphi}.
\end{proof}
Using a Gr\"onwall argument by invoking Proposition \ref{prop:evolution} and Proposition \ref{prop:naU}, Theorem \ref{thm:C1alpha} is proved.

\subsection{Propagation of higher regularity}
With the global-in-time propagation of $C^{1,\alpha}$ regularity established in the previous subsection, we now concern ourselves with the propagation of higher regularity. For the sake of clarity, and without loss of generality we take  $a_\pm = 1$.

To start with, we observe that Theorem \ref{thm:C1alpha} enables us to use the following Contour Dynamics Equation (see \cite{chu}): for all $t \ge 0$,
\begin{equation}
    \label{eq:LScontour}
    \begin{aligned}
        \pa_t z_+(\alpha,t) &= \frac{1}{2\pi}\int_0^{2\pi} \log|z_+(\alpha) - z_+(\beta)| \frac{\pa z_+}{\pa \beta}(\beta) d\beta\\
    &\qquad- \frac{1}{4\pi} \int_0^{2\pi} \tilde G(|z_+(\alpha) - z_+(\beta)|)\frac{\pa z_+}{\pa \beta}(\beta) d\beta\\
    &\qquad\qquad - \frac{1}{4\pi} \int_0^{2\pi} \tilde G(|z_+(\alpha) - z_-(\beta)|)\frac{\pa z_-}{\pa \beta}(\beta) d\beta,\\
    \pa_t z_-(\alpha,t) &= -\frac{1}{2\pi}\int_0^{2\pi} \log|z_-(\alpha) - z_-(\beta)| \frac{\pa z_-}{\pa \beta}(\beta) d\beta\\
    &\qquad+ \frac{1}{4\pi} \int_0^{2\pi} \tilde G(|z_-(\alpha) - z_+(\beta)|)\frac{\pa z_+}{\pa \beta}(\beta) d\beta\\
    &\qquad\qquad + \frac{1}{4\pi} \int_0^{2\pi} \tilde G(|z_-(\alpha) - z_-(\beta)|)\frac{\pa z_-}{\pa \beta}(\beta) d\beta,
    \end{aligned}
\end{equation}
equipped with initial parametrization $z_\pm(\alpha,0) = z_{\pm,0}(\alpha)$, $\alpha \in [0,2\pi]$. Here, the kernel $\tilde G: [0,\infty) \to \R$ defined by
\begin{equation}
    \tilde G(r) = \log r + K_0(r) - (\log 2 - \gamma_e),
\end{equation}
where $K_0$ is the modified Bessel function and $\gamma_e$ is the Euler's constant, enjoys the following properties: let $x \in \R^2$ and denote $r = |x|$.
\begin{enumerate}
    \item (Far-field properties) $\tilde G \in C^\infty([1,\infty))$. Moreover, the following estimates hold: for $r \ge 1$,
    \begin{equation}
        \label{tildeGfar}
        \begin{aligned}
            |\tilde G(r)| \le C_0 (1+\log r),\\
            |\na_x^j \tilde G(r)| \le \frac{C_j}{r^j},\quad j \in \mathbb{N},
        \end{aligned}
    \end{equation}
    where $C_0, C_j$ are universal constants.

    \item (Near-field properties) For $r \in (0,1)$, the following estimates hold:
    \begin{equation}
        \label{tildeGnear}
        \begin{aligned}
            |\tilde G(r)| \le C_0 r^2 (1+|\log r|),\quad |\na \tilde G(r)| \le C_1 r(1 + |\log r|),\\
            |\na^2 \tilde G(r)| \le C_2 (1 + |\log r|),\quad |\na^3 \tilde G(r)| \le C_3 \frac{1}{r}.
        \end{aligned}
    \end{equation}
    In particular, $\na \tilde G(r)$ is Log-Lipschitz.
\end{enumerate}
The bounds \eqref{tildeGfar} and \eqref{tildeGnear} are known properties of the Bessel potential (see e.g. \cite{bessel1}). Finally, we introduce the arc-chord quantity $$|z|_* := \inf_{\beta \neq \beta'} \left|\frac{z(\beta) - z(\beta')}{\beta - \beta'}\right|.$$

We dedicate the rest of the section to proving the following main result:

\begin{thm}
    \label{thm:C2alpha}
    Fix $\gamma \in (0,1)$. Assume that $z_{\pm,0} \in C^{2,\gamma}(\T)$, $|z_{\pm,0}|_* \ge m > 0$, and $|z_{\pm,0}'| \le M < \infty$. Then the system \eqref{eq:LScontour} admits a unique solution $z_\pm(\alpha,t) \in C^{2,\gamma}(\T)$ for all $t \ge 0$. In particular, for any $T > 0$, there exists $C(T)$ depending only on initial data, $T$, $m$, $M$, and $\gamma$ that 
    \begin{equation}
        \label{est:C2alphamain}
        \sup_{t \in [0,T]} (\|z_\pm(\cdot,t)\|_{C^{2,\gamma}(\T)} + |z_\pm|_*^{-1}  + |z_\pm'|_*)\le C(T).
    \end{equation}
\end{thm}

We start with showing the following bounds concerning a global arc-chord control and patch perimeter for $\sigma_\pm$ patches, which are consequences of the $C^{1,\alpha}$ theory. Such estimates are known for Euler patches, and we record them here for the sake of completeness.

\begin{prop}\label{prop:arc}
    For $\gamma \in (0,1)$, let $\sigma_\pm = a_\pm\chi_{\Omega_\pm(t)}$ be the global $C^{1,\gamma}$ patch solution to the left-handed model \eqref{eq:LS}. Let $\varphi_\pm \in C^{1,\gamma}$ be the patch-defining scalars corresponding to $\sigma_\pm$ satisfying \eqref{eq:phipm}. Define $\lambda_\pm(t) := |\p \Omega_\pm(t)|$ to be the perimeter of the $\sigma_\pm$ patches. For any $T > 0$, there exists $C(T) > 0$ depending only on initial data,  $\gamma$, and $T$ such that
    \begin{equation}
        \label{est:perimeter}
        \sup_{t \in [0,T]}\lambda_\pm(t) \le C(T).
    \end{equation}
    Moreover, let $z_\pm: [0,2\pi] \to \R^2$ be a $C^{1,\gamma}$ Lagrangian parametrization of the boundaries to the $\sigma_\pm$ patches. For any $T > 0$, there exists $C(T) > 0$ depending only on initial data,  $\gamma$, and $T$ such that
    \begin{equation}
        \label{est:arcchord}
        \sup_{t \in [0,T]} |z_\pm|_*^{-1} \le C(T).
    \end{equation}
\end{prop}
\begin{proof}
    Without loss of generality, we take $a_\pm = 1$. Fix $T > 0$. We start with proving \eqref{est:perimeter}. Since the scaling $\varphi_\pm \mapsto \mu_\pm \varphi_\pm$, $\mu_\pm > 0$, holds the patch boundary invariant, we may choose $\mu_\pm > 0$ appropriately such that the following gauge choice is valid: we select the initial parametrization $z_\pm: [0,2\pi] \to \R^2$ to be such that
    \begin{equation}
    \partial_\beta z_\pm(\beta, 0) = \nabla^\perp \varphi_\pm(z_\pm(\beta,0), 0).
    \end{equation}
    Note that this is indeed a valid regular parametrization since $|\na \varphi_\pm(\cdot,0)|_{\inf} > 0$. We claim that for any $t \ge 0$,
    \begin{equation}\label{perimeteraux1}
    \partial_\beta z_\pm(\beta, t) = \nabla^\perp \varphi_\pm(z_\pm(\beta,t), t).
    \end{equation}
    Indeed, on one hand, as $\pt z_\pm (\beta, t) = v_\pm(z_\pm(\beta,t),t)$, differentiating in $\beta$ we obtain
    $$
    \pt \pb z_\pm(\beta,t) = \na v_\pm(z_\pm(\beta,t),t) \pb z_\pm(\beta,t).
    $$
    On the other hand, from \eqref{eq:phipm} we have
    $$
    (\pt + v_\pm \cdot\na) \na^\perp \varphi_\pm = \na v_\pm \na^\perp \varphi_\pm,
    $$
    from which we obtain
    $$
    \pt \na^\perp \varphi_\pm(z_\pm(\beta,t),t) = \na v_\pm (z_\pm(\beta,t),t) \na^\perp \varphi_\pm (z_\pm(\beta,t),t).
    $$
    Then the claim follows from uniqueness for ODE. Hence,
    $$
    \lambda_\pm(t) = \int_0^{2\pi} |\pb z_\pm(\beta,t)|d\beta \le 2\pi \|\na \varphi_\pm\|_{L^\infty},
    $$
    which concludes the proof of \eqref{est:perimeter} after using \eqref{est:naphipminfty}, which implies that \\$\sup_{t \in [0,T]}\|\na \varphi_\pm(\cdot,t)\|_{L^\infty} < C(T)$ for some $C(T) > 0$.

    To show \eqref{est:arcchord}, we fix $\beta, \beta' \in [0,2\pi]$ such that $\beta \neq \beta'$. We see that from $\pt z_\pm (\beta, t) = v_\pm(z_\pm(\beta,t),t)$ there is:
    \begin{align*}
    \frac12 &\frac{d}{dt} |z_\pm (\beta, t)-z_\pm (\beta', t)|^2 \\
    &= (z_\pm (\beta, t)-z_\pm (\beta', t))\cdot (v_\pm(z_\pm(\beta,t),t)-v_\pm(z_\pm(\beta',t),t))\\
    &\ge - \|\na v_\pm (\cdot, t)\|_{L^\infty} |z_\pm (\beta, t)-z_\pm (\beta', t)|^2.
    \end{align*}
    From which we obtain
    $$
    \frac12 \frac{d}{dt} |z_\pm|_*^2 \ge -\|\na v_\pm (\cdot, t)\|_{L^\infty}|z_\pm|_*^2.
    $$
    The above differential inequality plus \eqref{est:navpminfty} yields \eqref{est:arcchord}.
\end{proof}

For convenience, we introduce the following multilinear maps: given bounded vectors $a, b$,
\begin{equation}
    \label{defn:I1}
    \calI_1^\pm(z;a) := \int_0^{2\pi} \na^i \tilde G(|z - z_\pm(\beta)|)a_i \pb z_\pm(\beta) d\beta,
\end{equation}
\begin{equation}
    \label{defn:I2}
    \calI_2^\pm (z;a,b) := \int_0^{2\pi} (\na^2 \tilde G(|z - z_\pm(\beta)|))^{ij}a_ib_j (\pb z_\pm(\beta) )d\beta,
\end{equation}
We first show the following  lemma.
\begin{lemma}
    Assume that $z_\pm \in C^1(\T)$, $|z|_* > 0$, and $a, b$ are bounded vectors. Then both maps $z \mapsto \calI_1^\pm(z;a)$ and $z \mapsto \calI_2^\pm(z;a,b)$ are uniformly bounded and Log-Lipschitz. In particular, the following bounds hold: for any $z \in \R^2$, there exists some universal constant $C > 0$ such that
    \begin{equation}
        \label{est:I1bound}
        |\calI_1^\pm(z;a)| \le C |a| \lambda_\pm(t),
    \end{equation}
    and
    \begin{equation}
        \label{est:I2bound}
        |\calI_2^\pm(z;a,b)| \le C(\|z_\pm\|_{C^1}, |z_\pm|_*)|a||b|,
    \end{equation}
    Moreover, for $|z_1 - z_2| < 1$, 
    \begin{equation}
        \label{est:I1LogLip}
        |\calI_1^\pm(z_1;a) - \calI_1^\pm(z_2;a)| \le C\lambda_\pm(t) |a| |z_1- z_2| \left(1+|\log |z_1 - z_2||\right),
    \end{equation}
    and
    \begin{equation}
        \label{est:I2LogLip}
        |\calI_2^\pm(z_1;a,b) - \calI_2^\pm(z_2;a,b)| \le C(\|z_\pm\|_{C^1}, |z_\pm|_*) |a||b| |z_1 - z_2|\left(1+|\log |z_1 - z_2||\right).
    \end{equation}
\end{lemma}
\begin{proof}
    We begin with estimates concerning $\calI_1$. The bound \eqref{est:I1bound} follows from the fact that $\na \tilde G$ is a uniformly bounded kernel in view of \eqref{tildeGfar} and \eqref{tildeGnear}. To show \eqref{est:I1LogLip}, we observe that
    \begin{align*}
        &\calI_1^\pm(z_1;a) - \calI_1^\pm(z_2;a)\\
        &\quad= \int_0^{2\pi} \left[\na \tilde G(|z_1 - z_\pm(\beta)|) -\na \tilde G(|z_2 - z_\pm(\beta)|)\right]\cdot a \pb z_\pm(\beta) d\beta.
    \end{align*}
    Using \eqref{tildeGfar}, \eqref{tildeGnear}, and that $\na \tilde G(r)$ is Log-Lipschitz, we arrive at
    \begin{align*}
        |\calI_1^\pm(z_1;a) - \calI_1^\pm(z_2;a)| \le C |a| |z_1- z_2| (1+\left|\log|z_1 - z_2|\right|)  \lambda_\pm(t),\quad |z_1 - z_2| < 1.
    \end{align*}
    This completes the proof of \eqref{est:I1LogLip}. 

    To prove \eqref{est:I2bound}, we define the set $E_r^\pm(z) := \{\beta \in [0,2\pi] \;|\; |z - z_\pm(\beta)| < r\}$ given $z \in \R^2$, $r > 0$. Then we decompose $\calI_2^\pm$ as follows:
    $$
    \calI_2^\pm (z;a,b) = \left(\int_{E_1^\pm(z)} + \int_{\T\backslash E_1^\pm(z)}\right)(\na^2 \tilde G(|z - z_\pm(\beta)|))^{ij}a_ib_j (\pb z_\pm(\beta) )d\beta =: I_{1} + I_{2}.
    $$
    By \eqref{tildeGfar}, we note that $\na^2 \tilde G(r)$ is uniformly bounded for $r \ge 1$. Then it follows that
    $$
    |I_2| \le C |a| |b| \lambda_\pm(t).
    $$
    To estimate $I_1$, we observe from \eqref{tildeGnear} that
    \begin{align*}
        |I_1| \le C\|z_\pm\|_{C^1} |a||b| \int_{E_1^\pm(z)} \left(1 + \log \frac{1}{|z - z_\pm(\beta)|}\right) d\beta.
    \end{align*}
    To proceed, we claim the following bound for the set $E_r^\pm(z)$: for any $z \in \R^2$, $r > 0$,
    \begin{equation}
        \label{est:Erz}
        |E_r^\pm(z)| \le \min\left\{2\pi, \frac{4r}{|z_\pm|_*}\right\}.
    \end{equation}
    To show \eqref{est:Erz}, we consider  $E_r(z) \neq \emptyset$. We choose $\beta_0 \in E_r^\pm(z)$. Then for any $\beta \in E_r^\pm(z)$, we use the definition of $|z_\pm|_*$  to deduce that
    $$
    |z_\pm|_* |\beta - \beta_0| \le |z_\pm(\beta) - z_\pm(\beta_0)| \le |z - z_\pm(\beta)| + |z - z_\pm(\beta_0)| \le 2r.
    $$
    Hence, $E_r^\pm(z) \subset \{\beta \in \T\;|\; |\beta - \beta_0| \le \frac{2r}{|z_\pm|_*}\}$ and the claimed estimate \eqref{est:Erz} follows. With \eqref{est:Erz} and the identity
    $$
    \log \frac{1}{r} = \int_0^\infty \chi_{\{r < e^{-t}\}} dt,\quad r \in (0,1),
    $$
    we deduce that
    \begin{equation}\label{est:logintegral}
    \begin{aligned}
        \int_{E_1^\pm (z)} \log \frac{1}{|z - z_\pm(\beta)|} d\beta &= \int_{E_1^\pm (z)} \int_0^\infty\chi_{E^\pm_{e^{-t}}(z)} dt d\beta \le 2\pi \int_0^\infty \min \{2\pi, \frac{4 e^{-t}}{|z_\pm|_*}\} dt\\
        &\le C\left(1 + \log \left(1 + \frac{1}{|z_\pm|_*}\right)\right).
    \end{aligned}
    \end{equation}
    Combining all estimates above, we conclude that
    $$
    |I_1| \le C(\|z_\pm\|_{C^1}, |z_\pm|_*)|a||b|.
    $$
    This estimate together with that for $I_2$ yields the desired bound \eqref{est:I2bound}.

    We finally prove \eqref{est:I2LogLip}. We fix $z_1, z_2 \in \R^2$ and let $\delta = |z_1 - z_2|$. Without loss of generality, we assume $\delta$ to be sufficiently small. Note that
    \begin{align*}
        \calI_2^\pm(z_1;a,b) - \calI_2^\pm(z_2;a,b) &= \left(\int_{E_{4\delta}^\pm(z_1)} + \int_{E_{1}^\pm(z_1)\backslash E_{4\delta}^\pm(z_1)} + \int_{\T\backslash E_{1}^\pm(z_1)}\right) \\
        &\left(\na^2 \tilde G(|z_1 - z_\pm(\beta)|) - \na^2 \tilde G(|z_2 - z_\pm(\beta)|)\right)^{ij} a_ib_j \pb z_\pm (\beta)d\beta\\
        &=: J_1 + J_2 + J_3.
    \end{align*}
    To estimate $J_1$, we note that for any $\beta \in E_{4\delta}^\pm(z_1)$, we must also have $|z_2 - z_\pm(\beta)| \le 5\delta$. Then invoking \eqref{tildeGnear} we have
    \begin{align*}
        |J_1| &\le C|a| |b| \|z_\pm\|_{C^1}\bigg(\int_{E_{4\delta}^\pm(z_1)} \left(1 + \log \frac{1}{|z_1 - z_\pm(\beta)|}\right) d\beta\\
        &\quad + \int_{E_{5\delta}^\pm(z_2)} \left(1 + \log \frac{1}{|z_2 - z_\pm(\beta)|}\right) d\beta\bigg).
    \end{align*}
    To estimate the first integral in the parenthesis above, we have
    \begin{equation}\label{est:deltalog}
    \begin{aligned}
        \int_{E_{4\delta}^\pm(z_1)} &\left(1 + \log \frac{1}{|z_1 - z_\pm(\beta)|}\right) d\beta\\
        &\le (1+ \log \frac{1}{4\delta})|E_{4\delta}^\pm(z_1)| + \int_{E_{4\delta}^\pm(z_1)} \log \frac{4\delta}{|z_1 - z_\pm(\beta)|} d\beta\\
        &\le C(|z_\pm|_*) \delta \log \delta^{-1} + \int_{E_{4\delta}^\pm(z_1)} \log \frac{4\delta}{|z_1 - z_\pm(\beta)|} d\beta,
    \end{aligned}
    \end{equation}
    where we used the bound \eqref{est:Erz} in the final inequality above. To estimate the integral in \eqref{est:deltalog}, a similar argument to \eqref{est:logintegral} reveals the bound:
    $$
    \int_{E_{4\delta}^\pm(z_1)} \log \frac{4\delta}{|z_1 - z_\pm(\beta)|} d\beta \le C\frac{\delta}{|z_\pm|_*} \left(1 - \log\frac{\delta}{|z_\pm|_*}\right).
    $$
    Hence the above bound together with \eqref{est:deltalog} yields:
    $$
    \int_{E_{4\delta}^\pm(z_1)} \left(1 + \log \frac{1}{|z_1 - z_\pm(\beta)|}\right) d\beta \le C(|z_\pm|_*)\delta(1+ \log \delta^{-1}).
    $$
    A similar argument also yields the following bound:
    $$
    \int_{E_{5\delta}^\pm(z_2)} \log \frac{4\delta}{|z_2 - z_\pm(\beta)|} d\beta \le C(|z_\pm|_*)\delta(1+ \log \delta^{-1}).
    $$
    Summarizing the above bounds, we conclude that
    \begin{equation}
        \label{est:J1}
        |J_1| \le C(|z_\pm|_*, \|z_\pm\|_{C^1})|a||b| \delta (1+\log \delta^{-1}).
    \end{equation}

    We then estimate $J_3$. Given $z_1, z_2$ fixed, we define $z(s) := z_1 + s(z_2 - z_1)$, $s \in [0,1]$. Then we observe that for $\beta \in \T \backslash E_1^\pm(z_1)$,
    $$
    |z(s) - z_\pm(\beta)| \ge \frac12,\quad s\in [0,1].
    $$
    Invoking the Mean Value Theorem and \eqref{tildeGfar}, we see that for any $\beta \in \T \backslash E_1^\pm(z_1)$, there exists $s(\beta) \in [0,1]$ such that
    \begin{equation}\label{est:J3}
    \begin{aligned}
        |J_3| &= \left|\int_{\T \backslash E_1^\pm(z_1)} \left[\na^3 \tilde G(z(s(\beta))-z_\pm(\beta))\right]^{ijk} a_ib_j (z_1 - z_2)_k \pb z_\pm(\beta) d\beta\right|\\
        &\le C|a||b| \int_{\T \backslash E_1^\pm(z_1)} \frac{\delta}{|z(s(\beta)) - z_\pm(\beta)|^3} |\pb z_\pm(\beta)| d\beta\\
        &\le C |a||b| \|z_\pm\|_{C^1} \delta.
    \end{aligned}
    \end{equation}

    Finally, to estimate $J_2$, we observe that for any $\beta \in E_{1}^\pm(z_1)\backslash E_{4\delta}^\pm(z_1)$, $|z(s) - z_\pm(\beta)| \ge \frac34 |z_1 - z_\pm(\beta)|$. Hence by a similar argument to \eqref{est:J3} using the Mean Value Theorem and \eqref{tildeGnear}, we have
    \begin{equation}\label{J2aux1}
        |J_2| \le C|a||b| \|z_\pm\|_{C^1} \delta \int_{4\delta < |z_1 - z_\pm(\beta)| \le 1} \frac{1}{|z_1 - z_\pm (\beta)|} d\beta.
    \end{equation}
    We define $E_j = \{\beta \in \T \;|\; 2^j 4\delta < |z_1 - z_\pm(\beta)| \le 2^{j+1}4\delta\}$. Let $J$ be the smallest natural number such that $2^{J+1} 4\delta \ge 1$. Then clearly $J \sim \log \delta^{-1}$. Hence,
    \begin{align*}
        \int_{4\delta < |z_1 - z_\pm(\beta)| \le 1} \frac{1}{|z_1 - z_\pm (\beta)|} d\beta &\le \sum_{j = 0}^{J} \int_{E_j} \frac{1}{|z_1 - z_\pm (\beta)|} d\beta \\
        &\le \sum_{j = 0}^{J} (2^{j} 4\delta)^{-1} |E_j| \le \sum_{j=0}^J (2^j 4\delta)^{-1} \frac{16\cdot 2^{j+1} \delta}{|z_\pm|_*}\\
        &\le \frac{C}{|z_\pm|_*} \log \delta^{-1},
    \end{align*}
    where we used that $E_j \subset E_{2^{j+1}4\delta}(z_1)$ and therefore $|E_j| \le \frac{16\cdot 2^{j+1} \delta}{|z_\pm|_*}.$
    Combining the above estimate with \eqref{J2aux1} we arrive at
    \begin{equation}
        \label{est:J2}
        |J_2| \le C(\|z_\pm\|_{C^1},|z_\pm|_* ) |a||b| \delta \log \delta^{-1}.
    \end{equation}
    The desired estimate \eqref{est:I2LogLip} then follows from \eqref{est:J1}, \eqref{est:J2}, and \eqref{est:J3}.
\end{proof}

With the technical preparations above, we complete the proof of Theorem \ref{thm:C2alpha}. Recalling \eqref{eq:LScontour}, we first write
\begin{equation}\label{eq:z+}
\pt z_+ = \calF_1 + \calF_2 + \calF_3,
\end{equation}
where
\begin{align*}
    \calF_1 &= \frac{1}{2\pi}\int_0^{2\pi} \log|z_+(\alpha) - z_+(\beta)| \frac{\pa z_+}{\pa \beta}(\beta) d\beta,\\
    \calF_2 &= - \frac{1}{4\pi} \int_0^{2\pi} \tilde G(|z_+(\alpha) - z_+(\beta)|)\frac{\pa z_+}{\pa \beta}(\beta) d\beta,\\
    \calF_3 &= - \frac{1}{4\pi} \int_0^{2\pi} \tilde G(|z_+(\alpha) - z_-(\beta)|)\frac{\pa z_-}{\pa \beta}(\beta) d\beta.
\end{align*}
Fix $T > 0$. From the global in time propagation of $C^{1,\gamma}$ boundary regularity for any $\gamma \in (0,1)$ and Proposition \ref{prop:arc}, there exists $C(T) > 0$ such that
\begin{equation}\label{geobound}
\sup_{t \in [0,T]}(\|z_\pm\|_{C^{1,\gamma}} + \lambda_\pm(t) + |z_\pm|_*^{-1}) < C(T).
\end{equation}
By \cite[Proposition 4.7]{bertozzi}, we  have the following bound:
$$
\|\calF_1(\cdot,t)\|_{C^{2,\gamma}} \le C(T) (1 + \|z_\pm\|_{C^{2,\gamma}}),\quad t \in [0,T].
$$
To estimate $\calF_2$, we observe that
$$
\pa_\alpha^2\calF_2(\alpha) = -\frac{1}{4\pi} \left(\calI_1^+(z_+(\alpha); \pa_\alpha^2 z_+(\alpha)) + \calI_2^+(z_+(\alpha); \p_\alpha z_+(\alpha),\p_\alpha z_+(\alpha))\right).
$$
Then applying \eqref{est:I1bound} and \eqref{est:I2bound}, we have
\begin{equation}\label{est:pa2F21}
\|\pa_\alpha^2\calF_2\|_{L^\infty} \le C(T)\|\pa_\alpha^2 z_\pm\|_{L^\infty} + C(T) \|z_\pm\|_{C^{1,\gamma}}^2 \le C(T) (1+ \|\pa_\alpha^2 z_\pm\|_{L^\infty}).
\end{equation}
Selecting $\alpha, \alpha' \in \T$ such that $\alpha \neq \alpha'$, we further observe that
\begin{align*}
    \p_\alpha^2 \calF_2(\alpha) &- \p_\alpha^2 \calF_2(\alpha') = -\frac{1}{4\pi}\left[\calI_1^+(z_+(\alpha); \pa_\alpha^2 z_+(\alpha)) - \calI_1^+(z_+(\alpha'); \pa_\alpha^2 z_+(\alpha))\right]\\
    &\quad -\frac{1}{4\pi}\left[\calI_1^+(z_+(\alpha'); \pa_\alpha^2 z_+(\alpha)-\pa_\alpha^2 z_+(\alpha'))\right]\\
    &\quad - \frac{1}{4\pi} \left[\calI_2^+(z_+(\alpha); \p_\alpha z_+(\alpha),\p_\alpha z_+(\alpha)) - \calI_2^+(z_+(\alpha'); \p_\alpha z_+(\alpha),\p_\alpha z_+(\alpha))\right]\\
    &\quad - \frac{1}{4\pi} \left[\calI_2^+(z_+(\alpha'); \p_\alpha z_+(\alpha)-\p_\alpha z_+(\alpha'),\p_\alpha z_+(\alpha))\right]\\
    &\quad - \frac{1}{4\pi} \left[\calI_2^+(z_+(\alpha'); \p_\alpha z_+(\alpha'),\p_\alpha z_+(\alpha)-\p_\alpha z_+(\alpha'))\right]\\
    &=: \sum_{j=1}^5 T_j.
\end{align*}
Denote $\delta = |\alpha - \alpha'|$ and assume without loss of generality that $\delta$ is sufficiently small. Using \eqref{geobound} together with \eqref{est:I1bound} and \eqref{est:I2bound} respectively, 
$$
|T_2| + |T_4| + |T_5| \le C(T)(\|\pa_\alpha^2 z_+\|_{C^\gamma} \delta^\gamma + \|z_+\|_{C^{1,\gamma}}\delta^\gamma) \le C(T) (1 + \|z_+\|_{C^{2,\gamma}})\delta^\gamma.
$$
Invoking \eqref{est:I1LogLip} and \eqref{est:I2LogLip} respectively, we also have
$$
|T_1| + |T_3| \le C(T) (\|\pa_\alpha^2 z_+\|_{C^\gamma} + \|z_+\|_{C^{1,\gamma}}^2)\delta \log \delta^{-1} \le C(T) (1 + \|\pa_\alpha^2 z_\pm\|_{C^{\gamma}}) \delta^\gamma.
$$
where the last inequality follows from choosing $\delta$ sufficiently small. Combining the above two estimates,
\begin{equation}
    \label{est:pa2F22}
    \|\pa_\alpha^2 \calF_2\|_{C^\gamma} \le C(T) (1 + \|\pa_\alpha^2 z_\pm\|_{C^{\gamma}}).
\end{equation}
A very similar argument also yields the following bound for $\calF_3$:
\begin{equation}
    \label{est:pa2F3}
    \|\pa_\alpha^2 \calF_3\|_{L^\infty} + \|\pa_\alpha^2 \calF_3\|_{C^\gamma} \le C(T) (1 + \|\pa_\alpha^2 z_\pm\|_{L^\infty} +\|\pa_\alpha^2 z_\pm\|_{C^{\gamma}}).
\end{equation}
Using \eqref{est:pa2F21}, \eqref{est:pa2F22}, \eqref{est:pa2F3}, and the equation for $z_+$, we arrive at the following differential inequality:
$$
\frac{d}{dt} (\|\pa_\alpha^2 z_+\|_{C^\gamma} + \|\pa_\alpha^2 z_+\|_{L^\infty}) \le C(T) (1 + \|\pa_\alpha^2 z_+\|_{C^\gamma} + \|\pa_\alpha^2 z_+\|_{L^\infty}),\quad t\in[0,T].
$$
from which we obtain $\sup_{t\in[0,T]} (\|\pa_\alpha^2 z_+\|_{C^\gamma} + \|\pa_\alpha^2 z_+\|_{L^\infty}) \le C(T)$. The same bound for $z_-$ is also obtained by a similar argument. This estimate together with \eqref{geobound} concludes the proof of Theorem \ref{thm:C2alpha}.

\section{The numerical algorithm}
\label{sec:nume1}

In this section, we validate numerically the theoretical finding of merger in the odd-even symmetry case, and explore potential mergers for other initial conditions. First, analytical representations are derived for four different initial conditions, as the basis for our numerical approximation. The approximation techniques we use are based on \cite{dritschel-89}. 

The initial boundaries of the vortex patches are closed curves in $\R^2$, which we parameterize clockwise for $\alpha \in [0, 2\pi]$. We linearly interpolate the curves with $N$ nodes, which correspond to our discretization of the boundaries of our patches. The nodes are initially evenly distributed in $\alpha$, and the contour is remeshed every $5$ time steps in order to maintain the uniform density of nodes (in $\alpha$). We then evolve the patches discretely using second-order finite differences to compute the differential contour element, second-order quadrature in $\alpha$, and a fourth-order Runge-Kutta scheme for the time evolution. 

In order to satisfy the CFL condition, we choose time-step: 
\begin{align}
    \Delta t = \frac{1}{N}
\end{align}

All numerical tests exhibit short-time convergence. We see no difference in behavior of patches evolved with $\Delta t = \frac{1}{N}$ and $\Delta t = \frac{1}{10N}$. As in \cite{dritschel-89}, we see that if there are too few nodes along the contour or if these nodes are spaced too far apart, the evolution of the discretized contour exhibits instability. After some time dependent on the initialization of the patches, the simulations become inaccurate as small scales form and the contour filaments. The numerical results we shall report are well before this time by an order of magnitude. In the results we report, the number of nodes is fixed to $N = 512$.

To visualize $F$ and $\phi$, we solve an elliptic problem in the Eulerian frame. We mesh a relevant rectangular domain containing the contours throughout their time evolution and compute the values over time of $\sigma_+, \sigma_-$ by evaluating for each time-step the  {scalar multiples by $a_\pm$ of} indicator functions on the patches $\Omega_{+}, \Omega_-$ in the interiors of $z_+, z_-$:
\begin{align*}
    \sigma_\pm = a_\pm \chi_{\Omega_{\pm}}(x)
\end{align*}
$F$ is computed by taking the average of $\sigma_+$ and $\sigma_-$. 

For the right-handed unscreened system, we then compute $\phi = (- \Delta)^{-1}F$ using the fundamental solution of the Laplacian: 
\begin{align}
    \phi(x,t) = -\frac{1}{2\pi}\int_{\R^2}\log |x-y| F(y) \, dy
\end{align}
For the right-handed screened system, we  compute $\phi = (I - \Delta)^{-1}F$. Analytically  we use the convolution of $F$ with the Bessel potential:
\begin{align}
    \phi(x, t) = \frac{1}{2\pi}\int_{\R^2} K_0(|x - y|) F(y,t)\,dy
\end{align}
where  $K_0$ is the modified Bessel function of the second kind of order zero.  Similarly, for the left-handed screened system we compute $\phi = ( \Delta - I)^{-1}F$, introducing a minus sign to the above. 

\subsection{Translational merger}
We start our numerical experimentation by looking at the right-handed, unscreened model in the case of the following initial conditions: set $\sigma_+(x,0) = \chi_{B}$ and $\sigma_-(x,0) = \chi_{\tilde B}$, where $B$ is a disk centered at $(d/2+r,0)$ with radius $r$. Here, $r, d > 0$. Also, $\tilde B$ is mirror image of $B$ across $x_2$ axis. Let $B$, $\tilde B$ be parametrized by $z_+, z_-$, evolved according to the following right-handed unscreened contour dynamics: 

\begin{equation}
    \label{eq:RU}
    \begin{aligned}
        \pt z_+ &= -\frac{1}{2\pi} \int_0^{2\pi} \log |z_+(\alpha) - z_-(\beta)| \pb z_-(\beta) d\beta,\\
        \pt z_- &= \frac{1}{2\pi} \int_0^{2\pi} \log |z_-(\alpha) - z_+(\beta)| \pb z_+(\beta) d\beta.
    \end{aligned}
\end{equation}

We use an explicit time-marching algorithm in order to advance the vortex patches every time-step. The contours of the vortex patches $z_+, z_-$ are discretized into $N$ nodes. Let $z_{\pm, j}$ represent the $j$-th node of each respective patch. We employ the trapezoidal rule to compute the integrals in \eqref{eq:RU}, and second-order central finite differences to approximate the differential contour elements $\pa_\beta z_\pm (\beta) d\beta$:
\begin{align}
    \Delta z_{\pm, j} = \frac{z_{\pm, j+1} - z_{\pm, j-1}}{2}
\end{align}
We then sum the kernel interactions. See:
\begin{align}
    \pa_t z_{+, i} \approx -\frac{1}{2\pi} \sum^N_{j=1} \log |z_{+, i} - z_{-, j}| \Delta z_{-, j} 
\end{align}
Let this finite differences method be denoted $V(z_+, z_-) = (\pa_tz_+, \pa_tz_-)$. To advance the position of the vortex patches in time, we utilize an explicit fourth-order Runge-Kutta method. For a given time step $\Delta t$, the positions at the next time are updated as follows:
\begin{equation}
\begin{aligned}
    (k_{1+}, k_{1-}) &= V\left(z_+^{(n)}, z_-^{(n)}\right) \\
       (k_{2+}, k_{2-}) &= V\left(z^{(n)}_+ + \frac{\Delta t}{2}k_{1+} \; , \; z^{(n)}_- + \frac{\Delta t}{2}k_{1-} \right) \\
    (k_{3+}, k_{3-}) &= V\left(z^{(n)}_+ + \frac{\Delta t}{2}k_{2+} \; , \; z^{(n)}_- + \frac{\Delta t}{2}k_{2-} \right) \\
    (k_{4+}, k_{4-}) &= V\left(z^{(n)}_+ + \Delta t k_{3+} \; , \; z^{(n)}_- +  \Delta tk_{3-} \right) \\
    z_{\pm}^{(n+1)} &= z^{(n)}_\pm + \frac{\Delta t}{6}(k_{1\pm} + 2k_{2\pm} + 2k_{3\pm} + k_{4\pm})
\end{aligned}
\end{equation}

As the vortex patches evolve, we remesh the contours $z_+, z_-$ every $5$ time-steps using a periodic cubic spline interpolation in order to maintain even spacing between the $N$ nodes.

\subsection{Antipodal disks}

The antipodal disk scenario is evolved according to the right-handed unscreened system \eqref{eq:RU} where the contours $z_\pm$ bound patches of opposite sign.  {That is, we set $\sigma_+(x,0) = \chi_{B}$ and $\sigma_-(x,0) = -\chi_{\tilde B}$, where $B$ is a disk centered at $(d/2+r,0)$ with radius $r$.} The contours $z_+, z_-$ evolve according to the following system:
\begin{equation}
    \label{eq:RUoppsign}
    \begin{aligned}
        \pt z_+ = \frac{1}{2\pi} \int_0^{2\pi} \log |z_+(\alpha) - z_-(\beta)| \pb z_-(\beta) d\beta,\\
        \pt z_- = \frac{1}{2\pi} \int_0^{2\pi} \log |z_-(\alpha) - z_+(\beta)| \pb z_+(\beta) d\beta.
    \end{aligned}
\end{equation}

The time integration, spatial discretization, and periodic remeshing protocols are identical to those described in Section 2.1. 

\subsection{Odd-even symmetry}
In \cite{chu}, a merger scenario is proved for the right-handed screened system with assumed odd-even symmetry. By odd-even symmetry, we mean the following symmetry: 
\begin{align} \label{eq:oddevensymm}
    \sigma_\pm (t,-x_1, x_2) = -\sigma_\pm (t,x_1,x_2),\quad \sigma_+(t,x_1,x_2) = \sigma_-(t,x_1,-x_2).
\end{align}
The four patches we track are bounded by contours $z, -z, \bar z, \tilde z$ where  $\tilde{z} = (-z_1, z_2)$ and $\bar{z} = (z_1,-z_2)$. The symmetry allows us to keep track of only one contour $z$ which is placed initially in the first quadrant, as the other three contours $-z, \bar z, \tilde z$ follow by symmetry. We evolve $z$ according to the following equation:
\begin{equation}
    \label{eq:RSoddodd}
    \begin{aligned}
        \pt z &= \frac{1}{4\pi} \int_0^{2\pi} \left(\log |z(\alpha) - z(\beta)| + K_0(|z(\alpha) - z(\beta)|)\right) \pb z(\beta) d\beta\\
        &\quad + \frac{1}{4\pi} \int_0^{2\pi} \left(\log |z(\alpha) - \tilde z(\beta)| + K_0(|z(\alpha) - \tilde z(\beta)|)\right) \pb \tilde z(\beta) d\beta\\
        &\quad\quad  + \frac{1}{4\pi} \int_0^{2\pi} \left(\log |z(\alpha) - \bar z(\beta)| - K_0(|z(\alpha) - \bar z(\beta)|)\right) \pb \bar z(\beta) d\beta\\
        &\quad\quad \quad  + \frac{1}{4\pi} \int_0^{2\pi} \left(\log |z(\alpha) + z(\beta)| - K_0(|z(\alpha) + z(\beta)|)\right) (-\pb  z(\beta)) d\beta
    \end{aligned}
\end{equation}

We also examine the right-handed unscreened system with the same odd-even symmetry. In the unscreened scenario, the contour $z$ is evolved according to the equation:
\begin{align}
    \label{eq:RUoddodd}
    \pt z = \frac{1}{2\pi} \int_0^{2\pi} \log|z(\alpha) - \bar z(\beta)| \pb\bar{z} (\beta )d\beta - \frac{1}{2\pi}\int_0^{2\pi} \log|z(\alpha) + z(\beta)| \pb{z} (\beta )d\beta
\end{align}

Our algorithm computes the time marching of $z$, which is then reflected across the $x-$ and $y-$axes to visualize the positions of all four patches over time. The discretization and remeshing are identical to those described in Section 2.1.

In order to compute $F$ and $\phi$ in these scenarios, we see that the symmetry \eqref{eq:oddevensymm} implies that:
\begin{align}
    F = \frac{1}{2}(\sigma_++\sigma_-) = \frac{1}{2}(\chi_{\Omega_{z}} -\chi_{\Omega_{\tilde z}} + \chi_{\Omega_{\bar z}} - \chi_{\Omega_{- z}})(x)
\end{align}
Then, $\phi$ is computed numerically as described above.

\subsection{Left-handed screened system.}
In Constantin--Hu \cite{chu}, merger is proven for the left-handed screened system modeled by the contour dynamics equation:
\begin{equation}
    \label{eq:LHS}
    \begin{aligned}
        \pt z_+ &= \frac{1}{2\pi} \int_0^{2\pi} \log |z_+(\alpha)- z_+(\beta)| \pb z_+ (\beta) d\beta \\ &\qquad+ \frac{1}{4\pi} \int_0^{2\pi} \bar G(|z_+(\alpha) - z_+(\beta)|)\pb z_+ (\beta) d\beta \\
        &\qquad \qquad+ \frac{1}{4\pi}\int_0^{2\pi}\bar G(|z_+(\alpha) - z_-(\beta)|)\pb z_-(\beta)d\beta, \\[0.5em]
        \pt z_- &= -\frac{1}{2\pi} \int_0^{2\pi} \log |z_-(\alpha)- z_-(\beta)| \pb z_- (\beta) d\beta\\
        &\qquad- \frac{1}{4\pi} \int_0^{2\pi} \bar G(|z_-(\alpha) - z_+(\beta)|)\pb z_+ (\beta) d\beta \\
        &\qquad \qquad- \frac{1}{4\pi}\int_0^{2\pi}\bar G(|z_-(\alpha) - z_-(\beta)|)\pb z_-(\beta)d\beta,
    \end{aligned}
\end{equation}
where $\bar G(r) = -K_0(r) - \log r$ and $K_0$ is the modified Bessel function of the second kind. The merger scenario is initiated by $\sigma_-$ a circular patch with intensity $1$ and $\sigma_+$ a Kirchoff's ellipse with intensity 1. See \cite{chu}. Our algorithm is the same as previously described.

\section{Numerical simulations}
\label{sec:nume2}

In our Figures, the leftmost plots show the positions of the contours $z_+, z_-$. The center plots show the position and value of $F$, with red patches corresponding to $\sigma_+$ and blue patches corresponding to $\sigma_-$. Bright red corresponds to $F$ taking on the value $0.5$, while dark red corresponds to $F = 1$, and blue corresponds to $F$ taking on the value $-0.5$ while dark blue corresponds to $F = -1$. The rightmost plots show the position and value of the level sets $\phi=c$  with a color corresponding to the certain value $c$. Blue corresponds to values of $\phi$ closer to its global minimum over time and space, while red curves correspond to values of $\phi$ closer to its global maximum.

\subsection{Translational merger}\label{sec:transmerge}

In this scenario of the right-handed unscreened system, the two contours $z_+, z_-$ travel in the negative $x_2$ direction together before merging. Their behavior is symmetric across the $y$-axis. In Figure \ref{fig:RHU} are contours $z_+, z_-$ initialized at $t=0$ as circles of radius $1$ centered at $(-2,0)$ and $(2,0)$ respectively. The level curves  drawn are
$$
\phi = -0.85, -0.75, -0.6, -0.5, -0.4, -0.3
$$
are drawn, going from  blue to red.

Here, we have numerical evidence of a situation in which the two contours and their magnetic fields initially have two connected components. As they translate in time, they become one connected component, therefore changing topology. Indeed, we see that numerically the change in topology of $F$ corresponds to a subsequent change in topology of $\phi$. Illustrating this is the disconnection of the level set $\phi = -0.4$ between $t=30$ and $t=40$.

The level set $\phi=-0.3$ shrinks and disappears throughout the time evolution: see that the level set $\phi=-0.3$ is present prior to $t=20$ and is not present at $t=30$ and therefore indicate another kind of change in topology of $\phi$. $\phi$ attains its global maximum at $t \approx  0.29$ and then decreases over time.

\subsection{Antipodal disks} \label{sec:antipodal} 
In this initialization of the right-handed unscreened system, two contours revolve about one another in the clockwise direction while deforming. This behavior is symmetric. The two contours  $z_+, z_-$ are initialized at $t=0$ as circles of radius $1$ centered at $(-1.5,0)$ and $(1.5,0)$ respectively, and throughout the simulation they do not merge or intersect. See Figure \ref{fig:antipodal}. Pictured are level curves $$\phi = -0.85, -0.75, -0.6, -0.5,-0.4, -0.27, -0.25, -0.23, -0.2$$ in colors blue to red, respectively. At $t=0$, the level set $\phi=-0.2$ is the only level set in two disconnected components. Although the contours remain disconnected, we observe a change in topology of $\phi$. Indeed, certain level curves of $\phi =c$ disconnect: the level set $\phi=-0.23$ is disconnected at $t=20$.

\subsection{Odd-even symmetry}\label{sec:oddeven}

We exhibit here the unscreened and screened odd-even symmetry systems with contours $z, -z, \tilde z, \bar z$ with $z$ initialized at $t=0$ as a circle of radius $0.45$ centered at $(0.5, 0.5)$. In both the unscreened and screened case, the contours drift towards one another and merge. See Figure \ref{fig:RHU_oddevensym} for the unscreened system. The level curves $$\phi=\pm 0.02,\;\pm0.04,\;\pm 0.05,\;\pm0.06$$ are pictured where negative values are blue and positive values are red. We observe changes in topology in $\phi$ over time; level sets $\phi=\pm0.06$ disconnect between $t=3$ and $t=4.5$.

As noted in Section 2.3, we numerically verify the existence of merger in the screened odd-even symmetry case for contours initialized close enough together, a small-scale version of which has been rigorously justified in \cite{chu}. See Figure \ref{fig:RHS_oddevensym}. The level curves for 
\begin{equation*}
    \phi=\pm 0.02,\; \pm0.04, \;\pm 0.05,\; \pm0.06,\; \pm0.99\phi_{\infty}
\end{equation*} 
are displayed where $\phi_{\infty}$ is the (discrete) maximum in space and time. Note that after $t=0$, the level curves $\phi=\pm0.99\phi_\infty$ vanish, exhibiting a topological change in $\phi$. Moreover, the level sets $\phi=\pm0.06$ become disconnected between $t=0$ and $t=1$ before vanishing, and the level sets $\phi=\pm0.05$ become disconnected between $t=2$ and $t=3$.

\subsection{Left-handed screened system}\label{sec:lefthanded}
We initialize $z_+$ as an ellipse and $z_-$ as a circle centered at the origin for small $R=1$ and $d=\frac R 2\equiv0.5$.
\begin{equation}
z_+(\alpha)=\bigl(d+2R+R\cos\alpha,\;2R\sin\alpha\bigr),
\; \;
z_-(\alpha)=\bigl(R\cos\alpha,\;R\sin\alpha\bigr),
\; \;\alpha\in[0,2\pi]
\end{equation}

Indeed, we see merger and un-merger of $z_+$ and $z_-$, as well as topological changes in $\phi$. See Figure \ref{fig:LHS_R1}. At $t\approx9$, another connected component of the level set $\phi=-0.25$ appears, pictured in the plot at $t=10$. At $t=15$ we see both connected components of the level set $\phi=-0.25$ have merged into one connected component. At $t=20$, we see that the level set $\phi=-0.25$ has again become two disconnected components. 

\begin{figure}[htpb]
    \noindent
    \centerline{
    \begin{tikzpicture}
    
    \node[inner sep=0pt, anchor=center] (O10T0) at (current page.center)
    {\includegraphics[width=0.5\paperwidth]{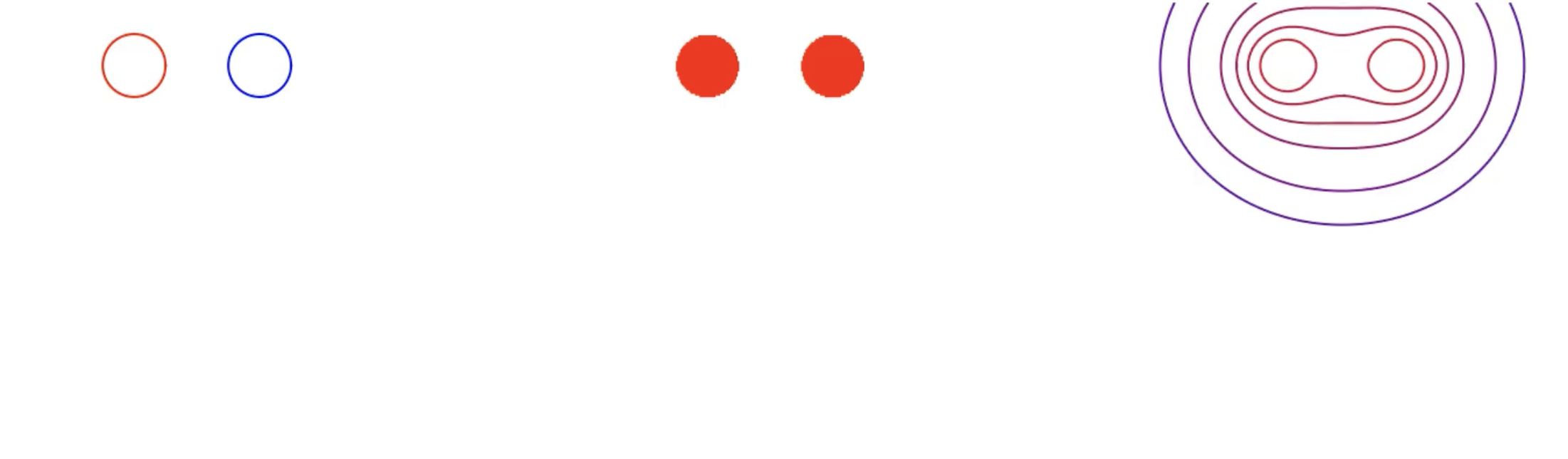}};
    
    \node[inner sep=0pt, anchor=center] (O20T0) at ([yshift=\ygap] O10T0.south) {\includegraphics[width=0.5\paperwidth]{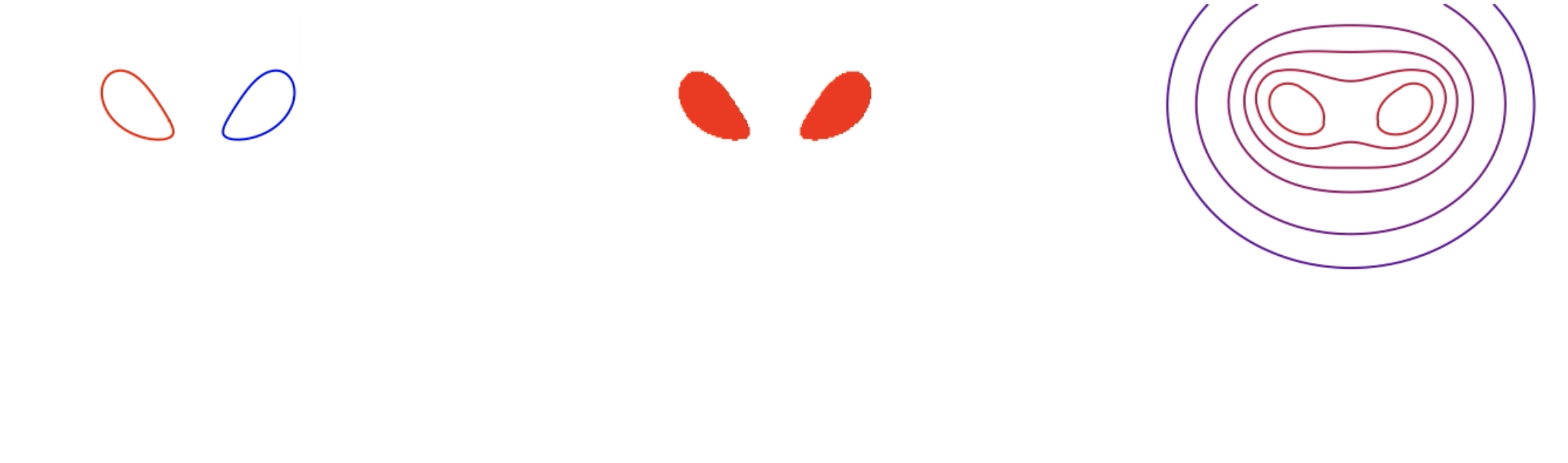}};

    \node[inner sep=0pt, anchor=center] (O30T0) at ([yshift=\ygap] O20T0.south) {\includegraphics[width=0.5\paperwidth]{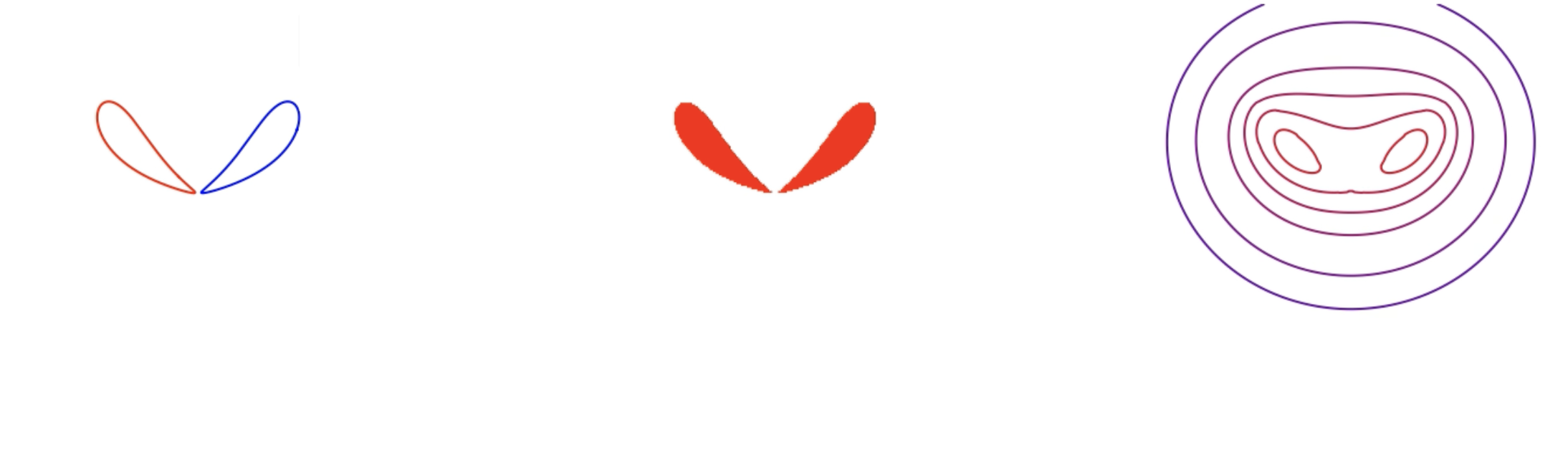}};

    \node[inner sep=0pt, anchor=center] (O40T0) at ([yshift=\ygap] O30T0.south) {\includegraphics[width=0.5\paperwidth]{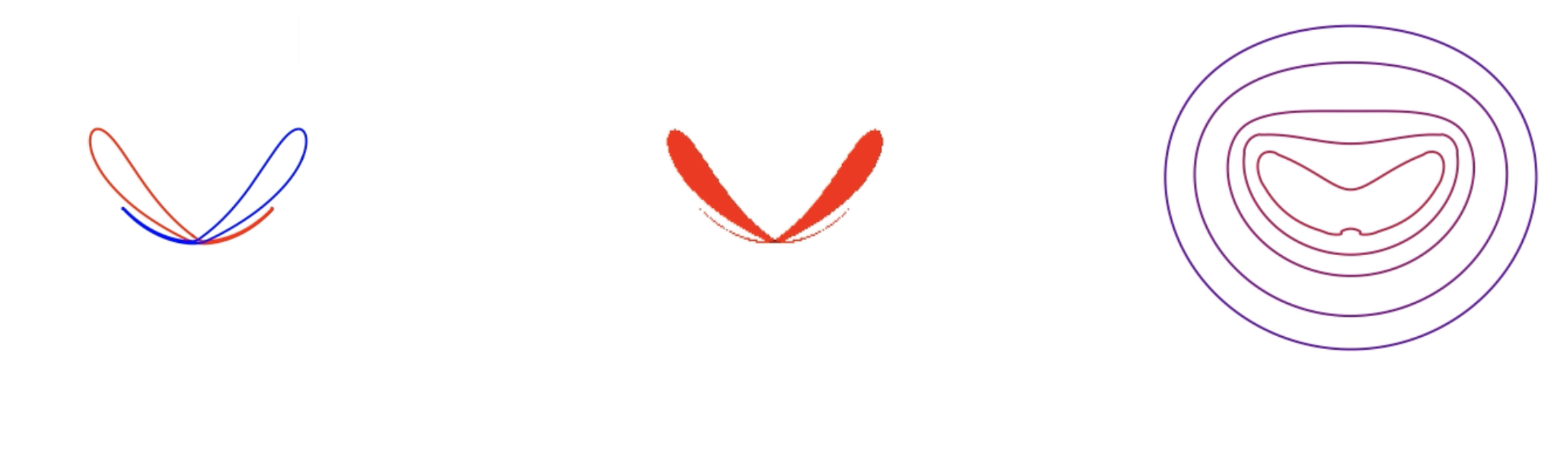}};

    \node[inner sep=0pt, anchor=center] (O50T0) at ([yshift=\ygap] O40T0.south) {\includegraphics[width=0.5\paperwidth]{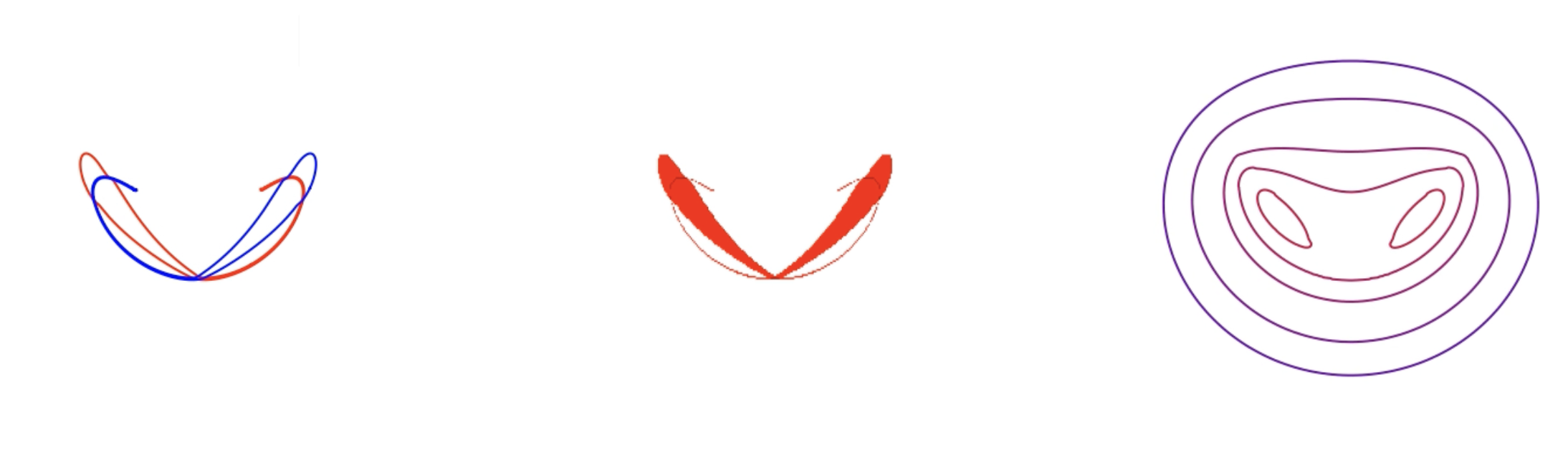}};

    \node (LO10) at ([xshift=-0.1mm]O10T0.west) {\rotatebox[origin=c]{90}{$t=0$}};
    \node (LO20) at ([xshift=-0.1mm]O20T0.west) {\rotatebox[origin=c]{90}{$t=10$}};
    \node (LO30) at ([xshift=-0.1mm]O30T0.west) {\rotatebox[origin=c]{90}{$t=20$}};
    \node (LO40) at ([xshift=-0.1mm]O40T0.west) {\rotatebox[origin=c]{90}{$t=30$}};
    \node (LO50) at ([xshift=-0.1mm]O50T0.west) {\rotatebox[origin=c]{90}{$t=40$}};

    \node[anchor=south, xshift=0mm, yshift=2mm] at ($(O10T0.north west)!0.1666!(O10T0.north east)$) {$z_\pm$};
    \node[anchor=south,  xshift=0mm, yshift=2mm] at ($(O10T0.north west)!0.5000!(O10T0.north east)$) {$F$};
    \node[anchor=south, xshift=0mm, yshift=2mm] at ($(O10T0.north west)!0.8333!(O10T0.north east)$) {$\phi$ };
    
    \end{tikzpicture} }
    \caption{Right-handed unscreened system with the same sign initial data in Section \ref{sec:transmerge}.}
    \label{fig:RHU}
\end{figure} 
\begin{figure}[htbp]
    \noindent
    \centerline{
    \begin{tikzpicture}
    
    \node[inner sep=0pt, anchor=center] (O10T0) at (current page.center)
    {\includegraphics[width=0.5\paperwidth]{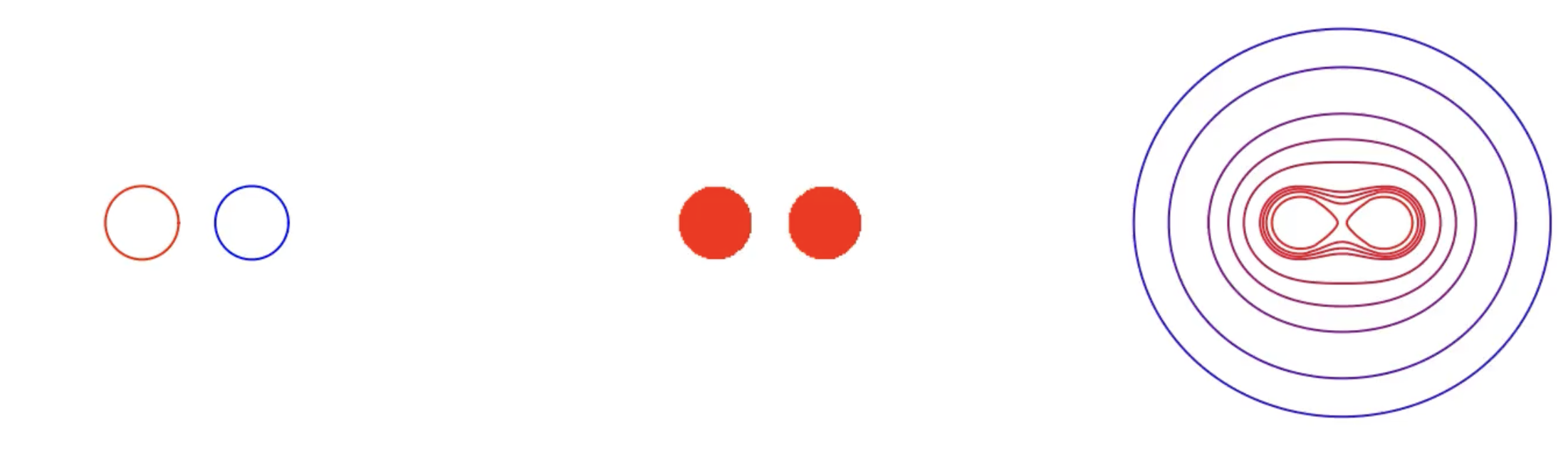}};
    
    \node[inner sep=0pt, anchor=center] (O20T0) at ([yshift=\ygap] O10T0.south) {\includegraphics[width=0.5\paperwidth]{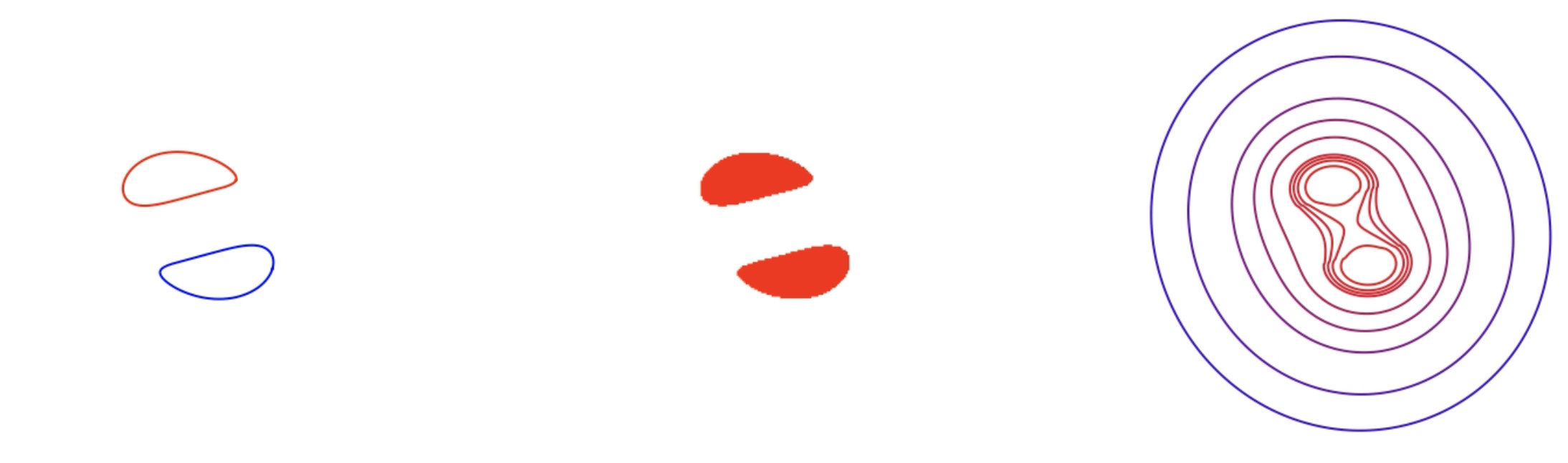}};

    \node[inner sep=0pt, anchor=center] (O30T0) at ([yshift=\ygap] O20T0.south) {\includegraphics[width=0.5\paperwidth]{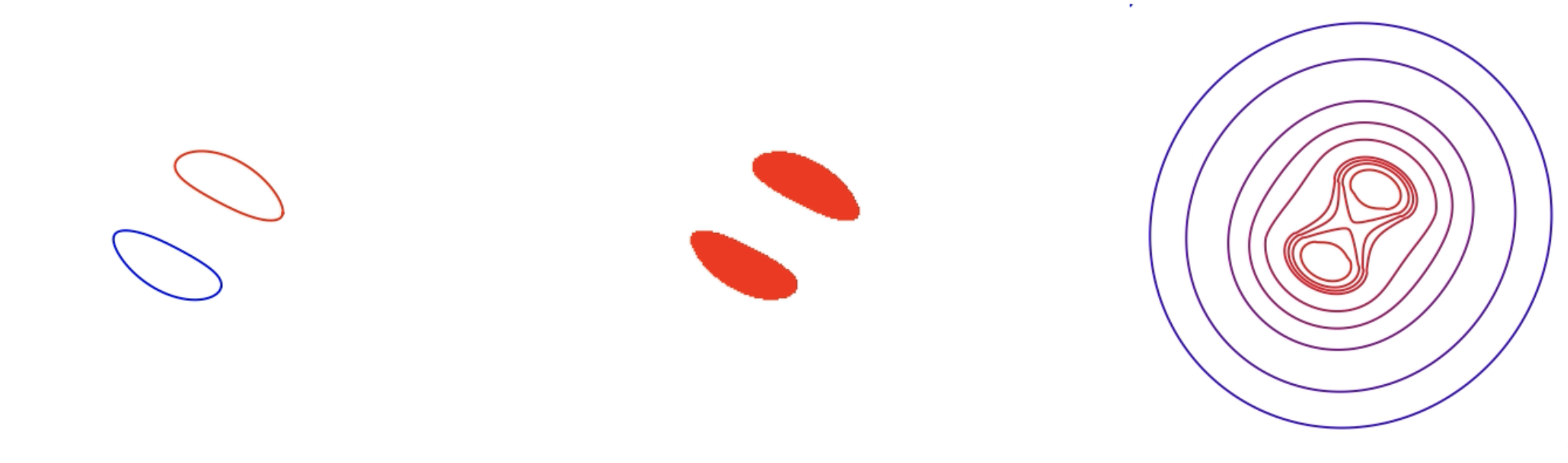}};

    \node[inner sep=0pt, anchor=center] (O40T0) at ([yshift=\ygap] O30T0.south) {\includegraphics[width=0.5\paperwidth]{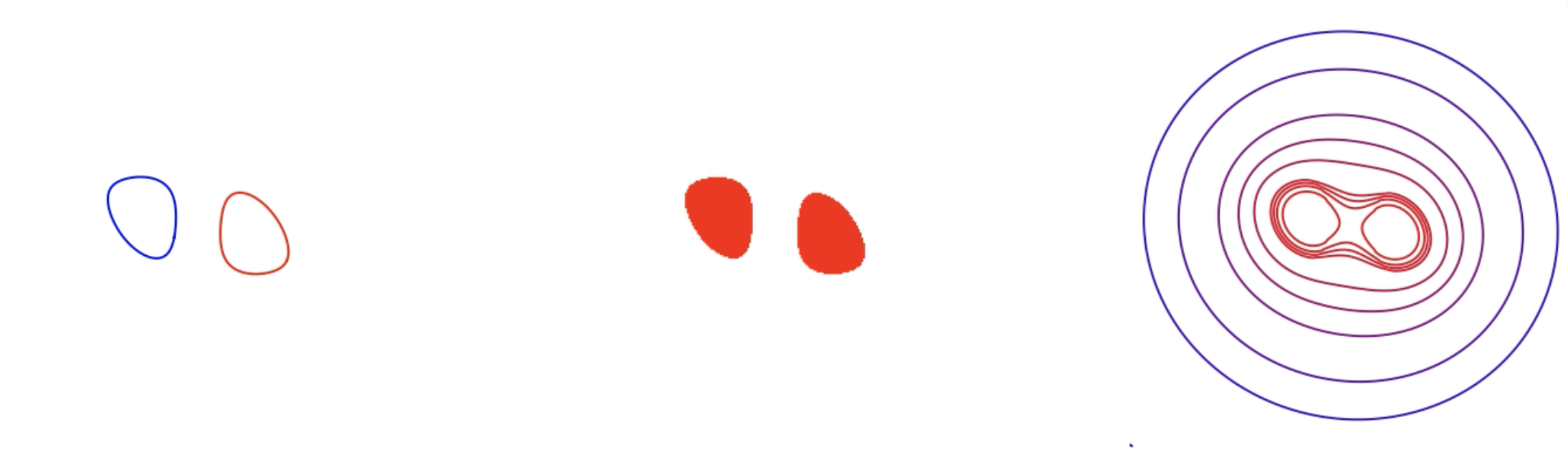}};

    \node[inner sep=0pt, anchor=center] (O50T0) at ([yshift=\ygap] O40T0.south) {\includegraphics[width=0.5\paperwidth]{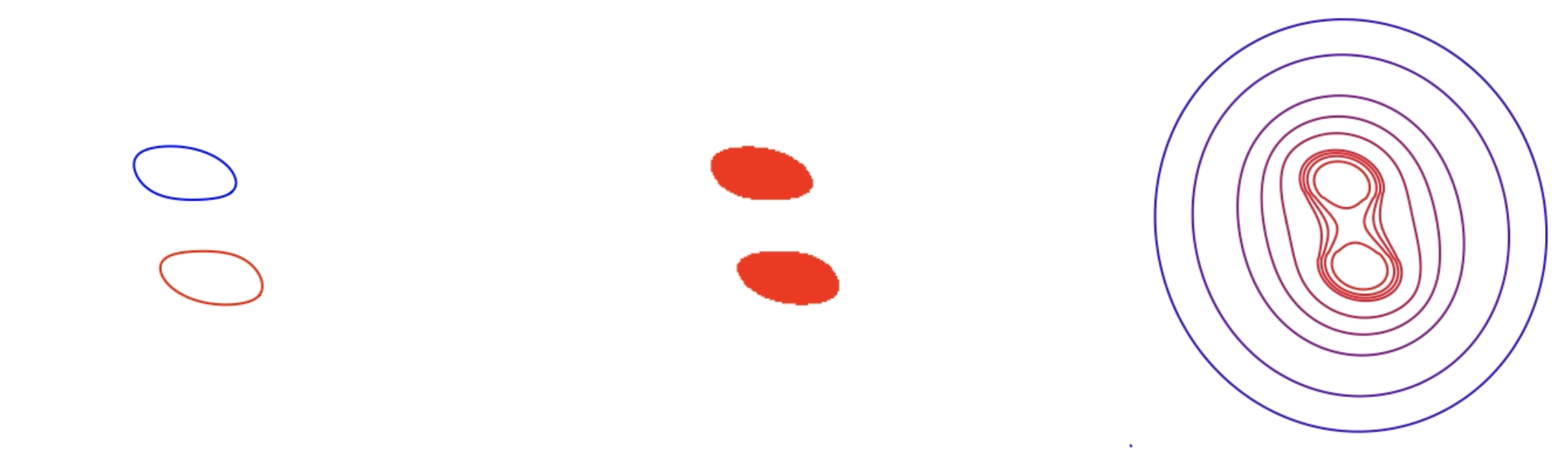}};

    \node (LO10) at ([xshift=-0.1mm]O10T0.west) {\rotatebox[origin=c]{90}{$t=0$}};
    \node (LO20) at ([xshift=-0.1mm]O20T0.west) {\rotatebox[origin=c]{90}{$t=10$}};
    \node (LO30) at ([xshift=-0.1mm]O30T0.west) {\rotatebox[origin=c]{90}{$t=20$}};
    \node (LO40) at ([xshift=-0.1mm]O40T0.west) {\rotatebox[origin=c]{90}{$t=30$}};
    \node (LO50) at ([xshift=-0.1mm]O50T0.west) {\rotatebox[origin=c]{90}{$t=40$}};

    \node[anchor=south, xshift=0mm, yshift=1mm] at ($(O10T0.north west)!0.1666!(O10T0.north east)$) {$z_\pm$};
    \node[anchor=south,  xshift=0mm, yshift=1mm] at ($(O10T0.north west)!0.5000!(O10T0.north east)$) {$F$};
    \node[anchor=south, xshift=0mm, yshift=1mm] at ($(O10T0.north west)!0.8333!(O10T0.north east)$) {$\phi$};
    
    \end{tikzpicture}}
    \caption{Right-handed unscreened system the opposite sign initial data in Section \ref{sec:antipodal}.}
    \label{fig:antipodal}
\end{figure}
\begin{figure}[htbp]
    \noindent
    \centerline{
    \begin{tikzpicture}
    
    \node[inner sep=0pt, anchor=center] (O10T0) at (current page.center)
    {\includegraphics[width=0.5\paperwidth]{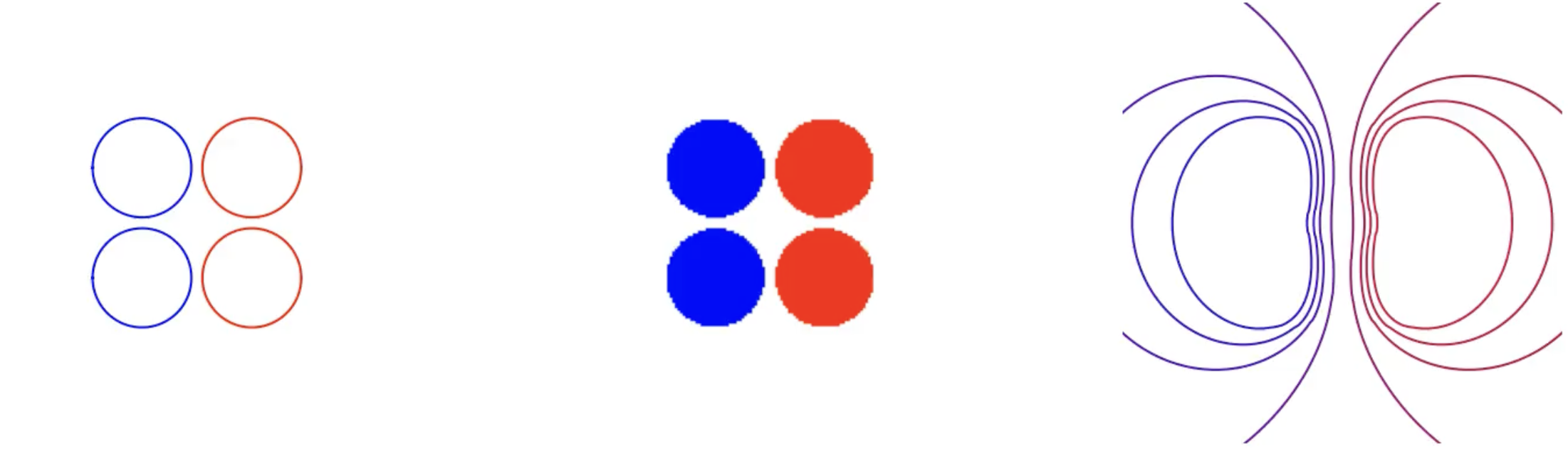}};
    
    \node[inner sep=0pt, anchor=center] (O20T0) at ([yshift=\ygap] O10T0.south) {\includegraphics[width=0.5\paperwidth]{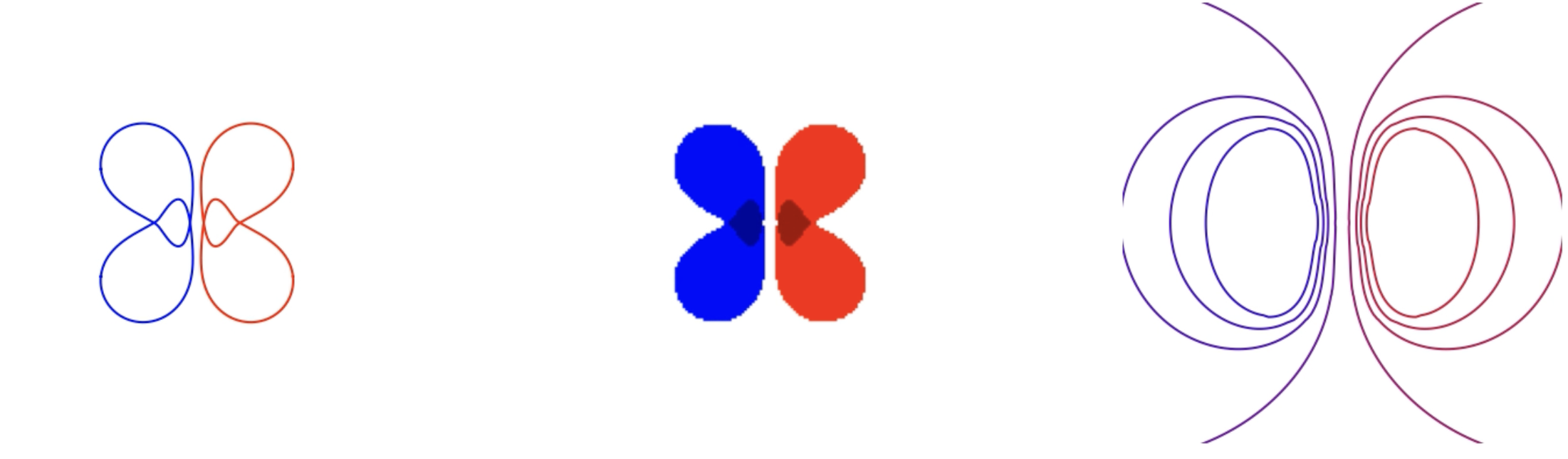}};

    \node[inner sep=0pt, anchor=center] (O30T0) at ([yshift=\ygap] O20T0.south) {\includegraphics[width=0.5\paperwidth]{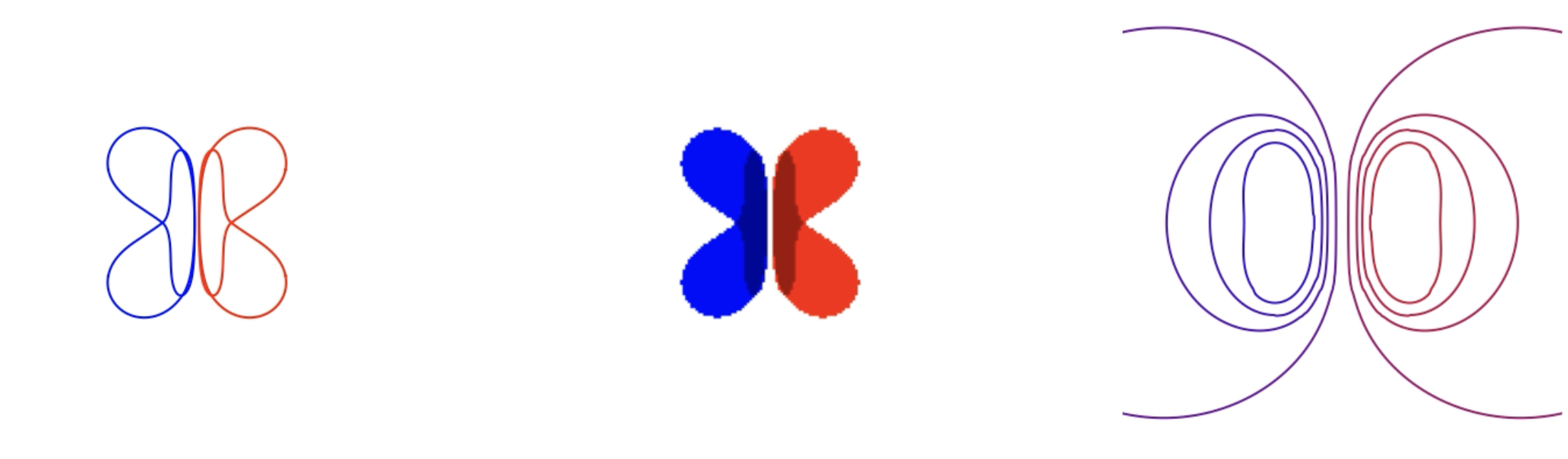}};

    \node[inner sep=0pt, anchor=center] (O40T0) at ([yshift=\ygap] O30T0.south) {\includegraphics[width=0.5\paperwidth]{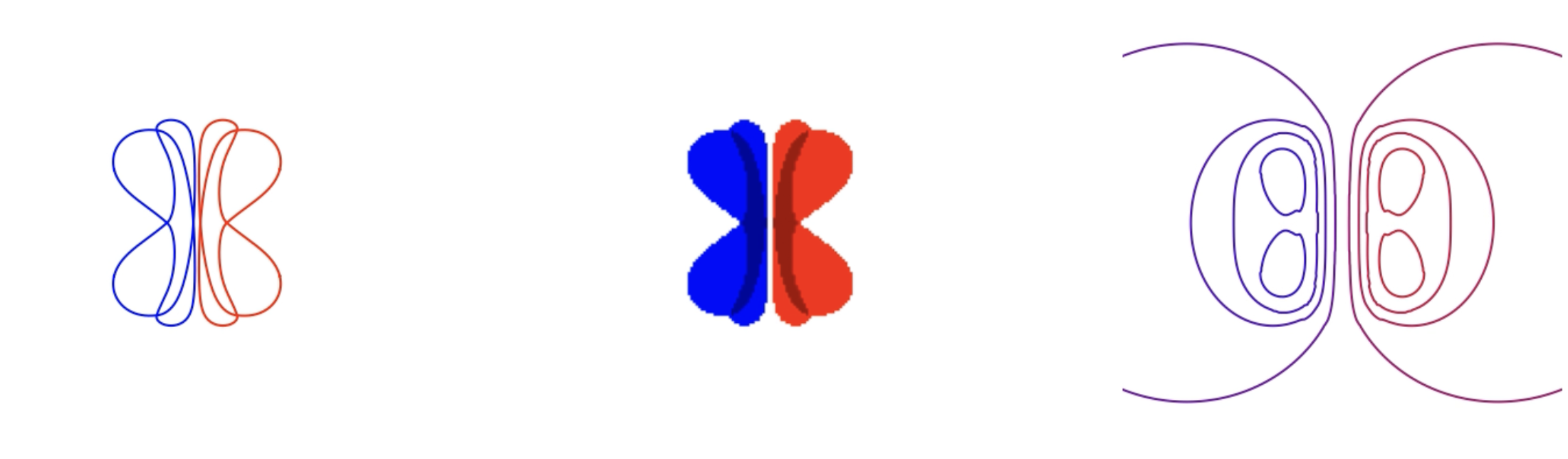}};

    \node[inner sep=0pt, anchor=center] (O50T0) at ([yshift=\ygap] O40T0.south) {\includegraphics[width=0.5\paperwidth]{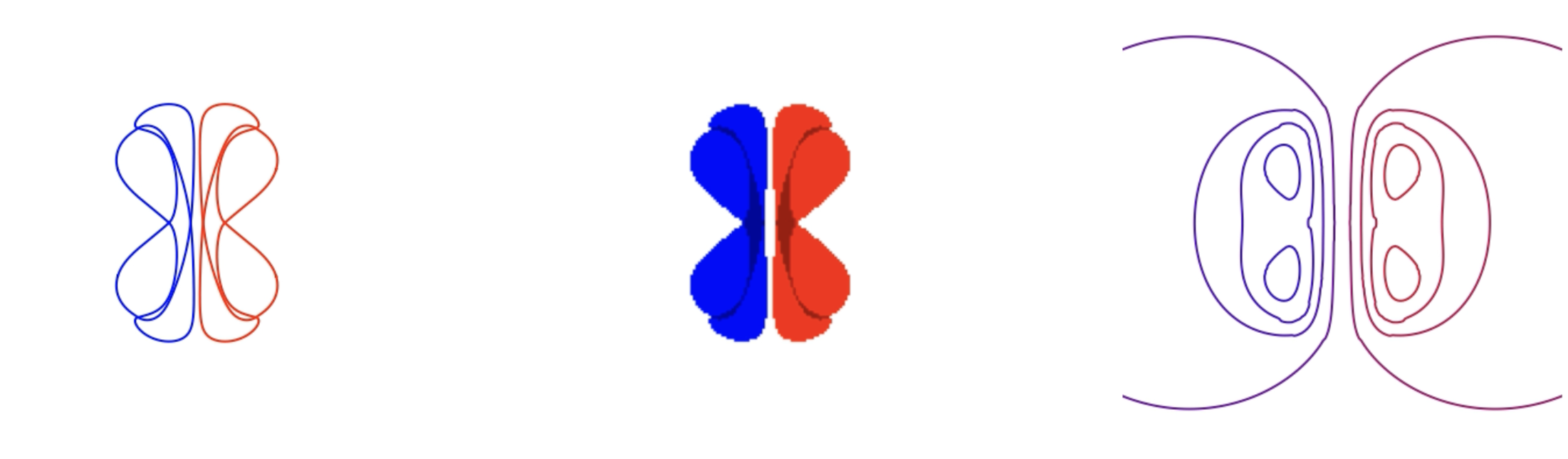}};

    \node (LO10) at ([xshift=-0.1mm]O10T0.west) {\rotatebox[origin=c]{90}{$t=0$}};
    \node (LO20) at ([xshift=-0.1mm]O20T0.west) {\rotatebox[origin=c]{90}{$t=1.5$}};
    \node (LO30) at ([xshift=-0.1mm]O30T0.west) {\rotatebox[origin=c]{90}{$t=3$}};
    \node (LO40) at ([xshift=-0.1mm]O40T0.west) {\rotatebox[origin=c]{90}{$t=4.5$}};
    \node (LO50) at ([xshift=-0.1mm]O50T0.west) {\rotatebox[origin=c]{90}{$t=6$}};

    \node[anchor=south, xshift=0mm, yshift=1mm] at ($(O10T0.north west)!0.1666!(O10T0.north east)$) {$z, \tilde{z},\bar{z},-z$};
    \node[anchor=south,  xshift=0mm, yshift=1mm] at ($(O10T0.north west)!0.5000!(O10T0.north east)$) {$F$};
    \node[anchor=south, xshift=0mm, yshift=1mm] at ($(O10T0.north west)!0.8333!(O10T0.north east)$) {$\phi$};
    
    \end{tikzpicture}}
    \caption{Right-handed unscreened system with odd-even symmetric data given in Section \ref{sec:oddeven}.}
    \label{fig:RHU_oddevensym}
\end{figure}

\begin{figure}[htbp]
    \noindent
    \centerline{
    \begin{tikzpicture}
    
    \node[inner sep=0pt, anchor=center] (O10T0) at (current page.center)
    {\includegraphics[width=0.5\paperwidth]{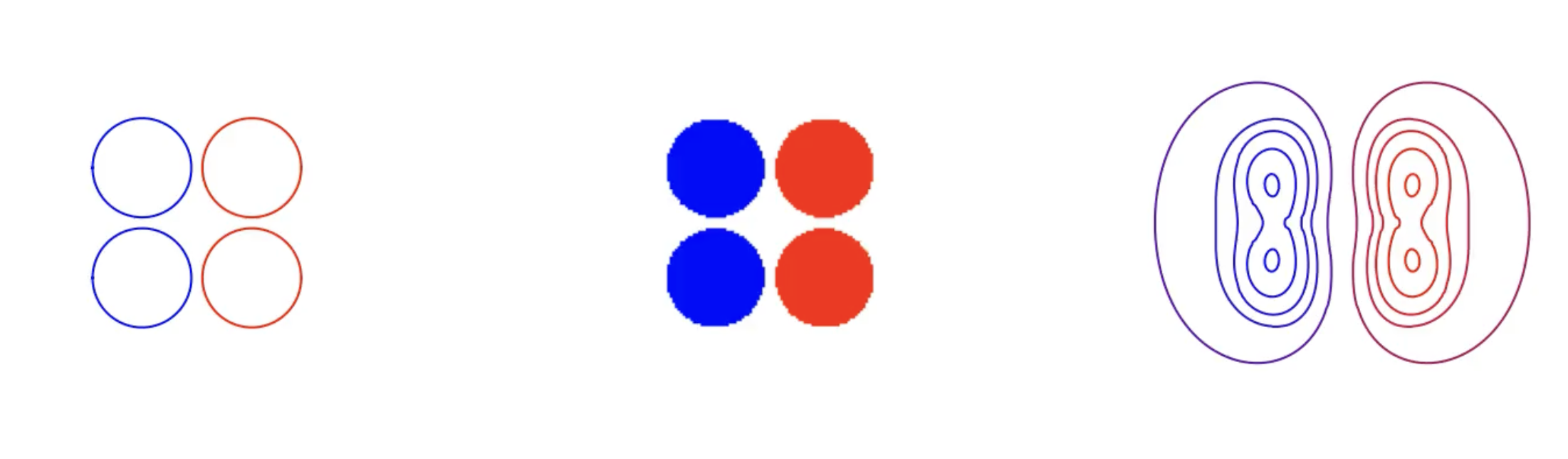}};
    
    \node[inner sep=0pt, anchor=center] (O20T0) at ([yshift=\ygap] O10T0.south) {\includegraphics[width=0.5\paperwidth]{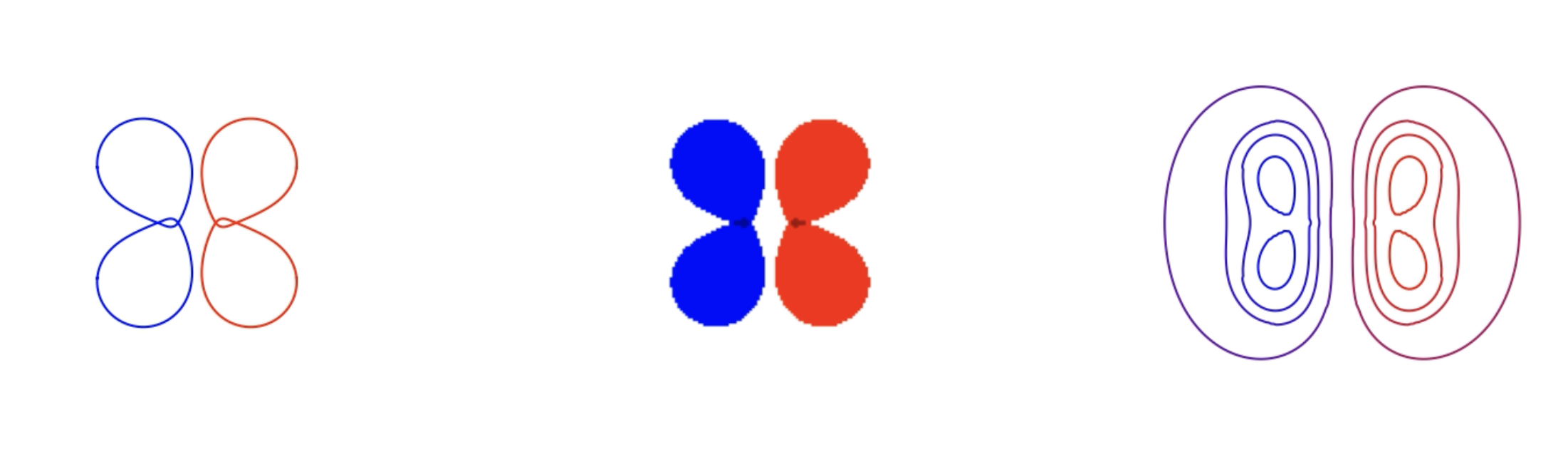}};

    \node[inner sep=0pt, anchor=center] (O30T0) at ([yshift=\ygap] O20T0.south) {\includegraphics[width=0.5\paperwidth]{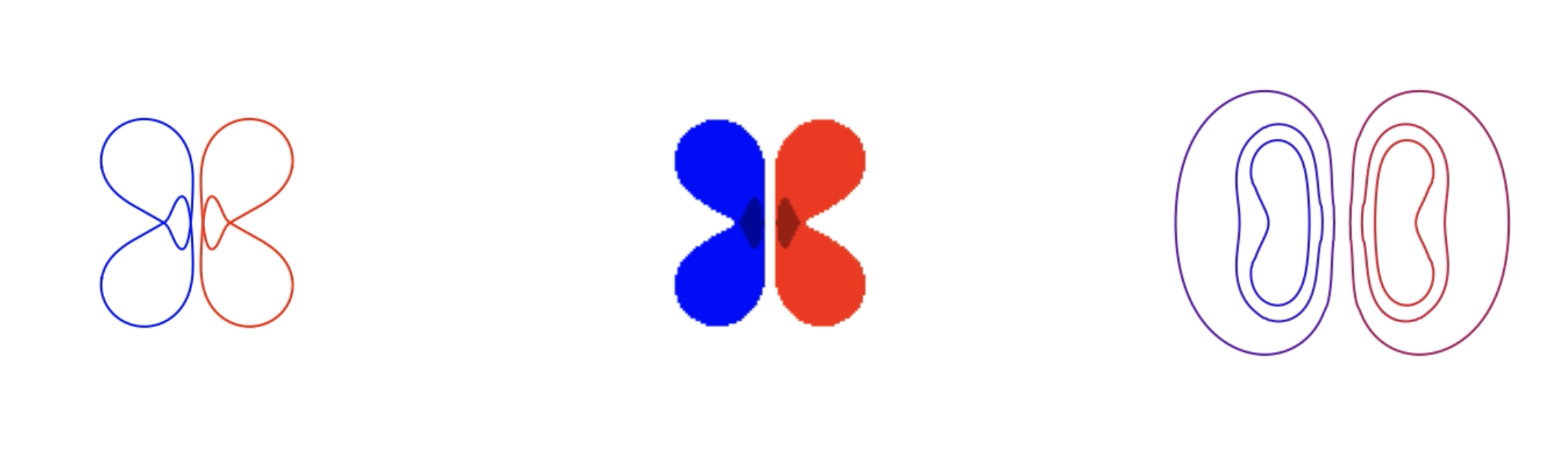}};

    \node[inner sep=0pt, anchor=center] (O40T0) at ([yshift=\ygap] O30T0.south) {\includegraphics[width=0.5\paperwidth]{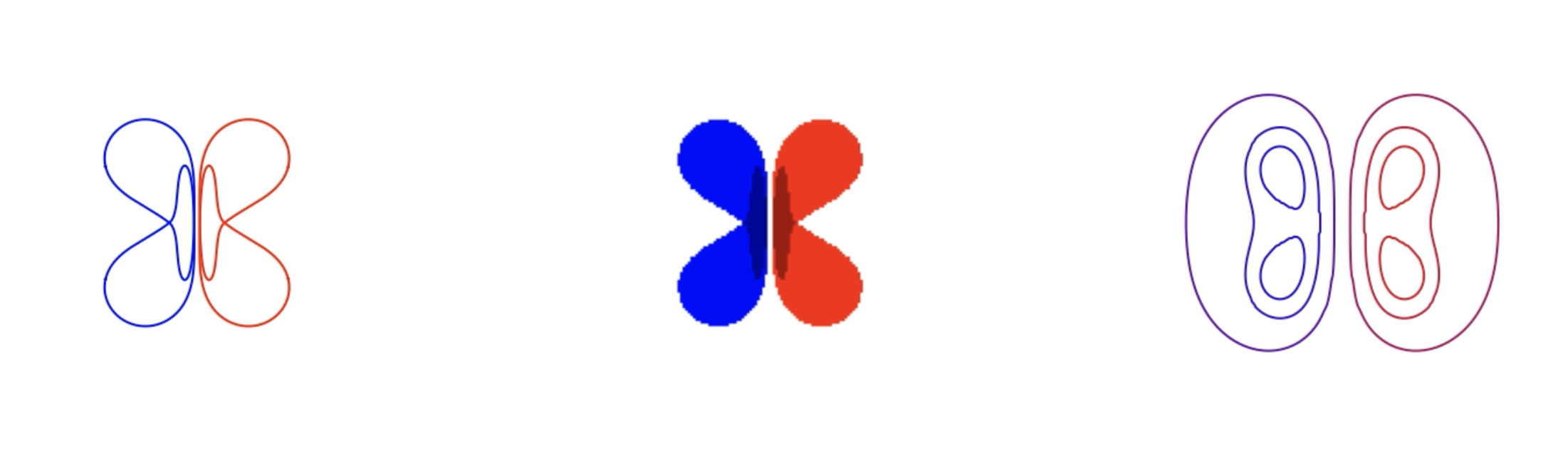}};

    \node[inner sep=0pt, anchor=center] (O50T0) at ([yshift=\ygap] O40T0.south) {\includegraphics[width=0.5\paperwidth]{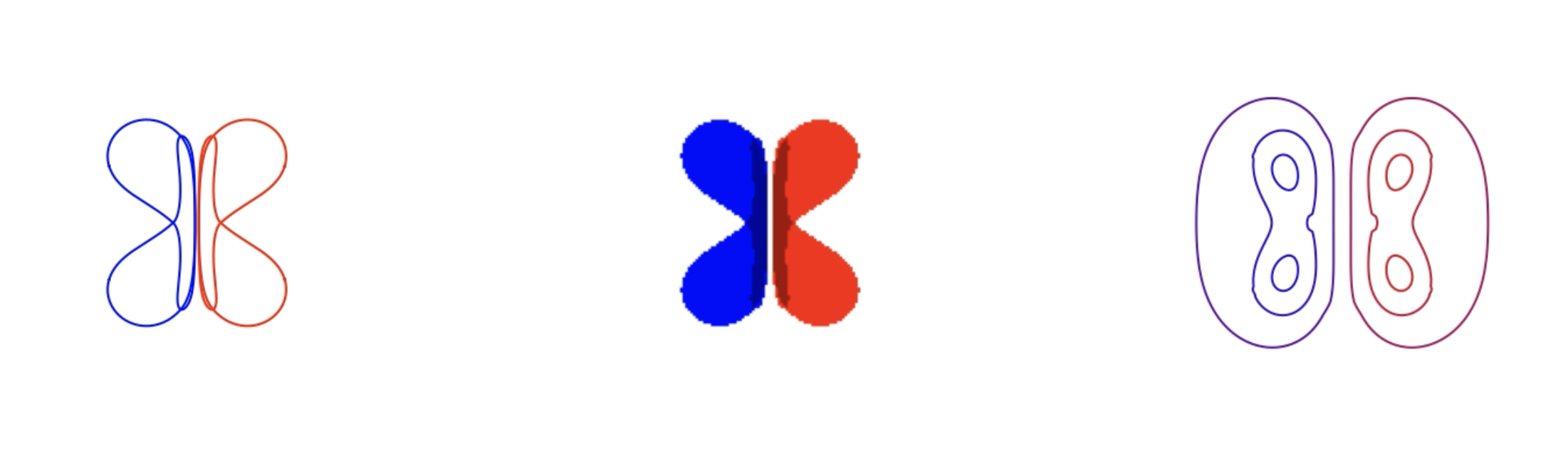}};

    \node (LO10) at ([xshift=-0.1mm]O10T0.west) {\rotatebox[origin=c]{90}{$t=0$}};
    \node (LO20) at ([xshift=-0.1mm]O20T0.west) {\rotatebox[origin=c]{90}{$t=1$}};
    \node (LO30) at ([xshift=-0.1mm]O30T0.west) {\rotatebox[origin=c]{90}{$t=2$}};
    \node (LO40) at ([xshift=-0.1mm]O40T0.west) {\rotatebox[origin=c]{90}{$t=3$}};
    \node (LO50) at ([xshift=-0.1mm]O50T0.west) {\rotatebox[origin=c]{90}{$t=4$}};

    \node[anchor=south, xshift=0mm, yshift=1mm] at ($(O10T0.north west)!0.1666!(O10T0.north east)$) {$z, \tilde{z},\bar{z},-z$};
    \node[anchor=south,  xshift=0mm, yshift=1mm] at ($(O10T0.north west)!0.5000!(O10T0.north east)$) {$F$};
    \node[anchor=south, xshift=0mm, yshift=1mm] at ($(O10T0.north west)!0.8333!(O10T0.north east)$) {$\phi$};
    
    \end{tikzpicture}}
    \caption{Right-handed screened system with odd-even symmetric initial given in Section \ref{sec:oddeven}.}
    \label{fig:RHS_oddevensym}
\end{figure}
\begin{figure}[htbp]
    \noindent
    \centerline{
    \begin{tikzpicture}
    
    \node[inner sep=0pt, anchor=center] (O10T0) at (current page.center)
    {\includegraphics[width=0.5\paperwidth]{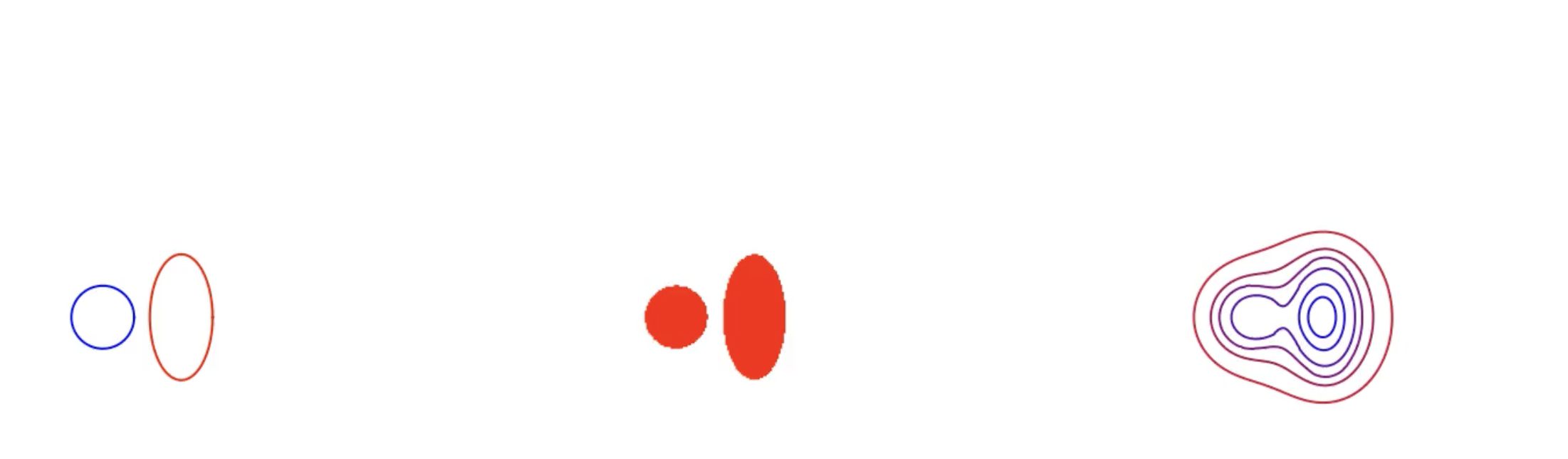}};
    
    \node[inner sep=0pt, anchor=center] (O20T0) at ([yshift=\ygap] O10T0.south) {\includegraphics[width=0.5\paperwidth]{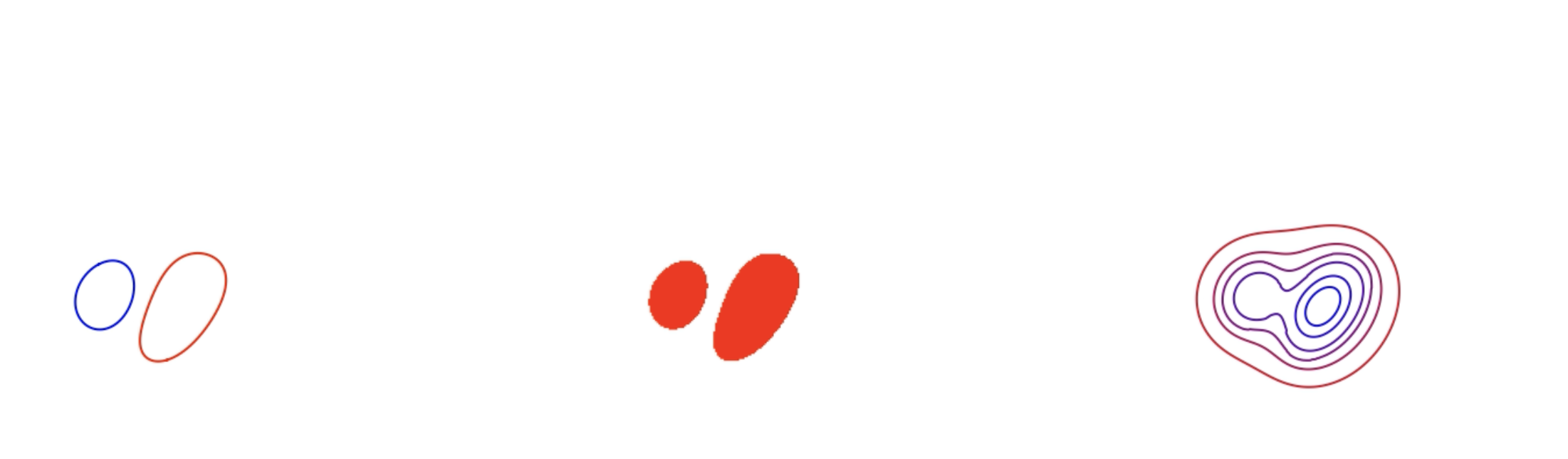}};

    \node[inner sep=0pt, anchor=center] (O30T0) at ([yshift=\ygap] O20T0.south) {\includegraphics[width=0.5\paperwidth]{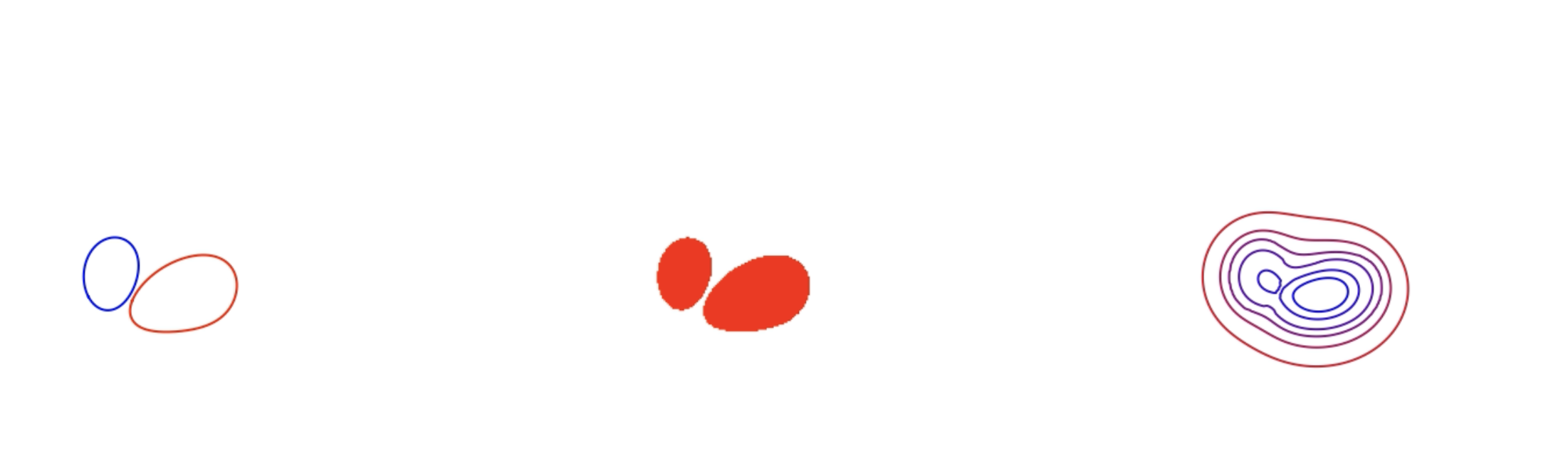}};

    \node[inner sep=0pt, anchor=center] (O40T0) at ([yshift=\ygap] O30T0.south) {\includegraphics[width=0.5\paperwidth]{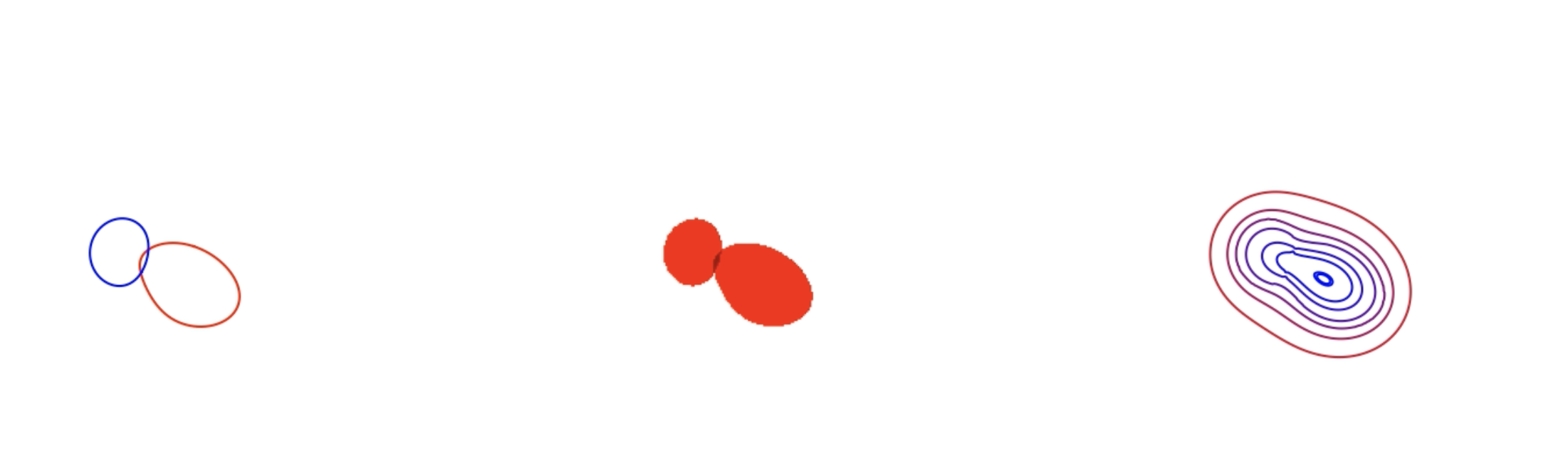}};

    \node[inner sep=0pt, anchor=center] (O50T0) at ([yshift=\ygap] O40T0.south) {\includegraphics[width=0.5\paperwidth]{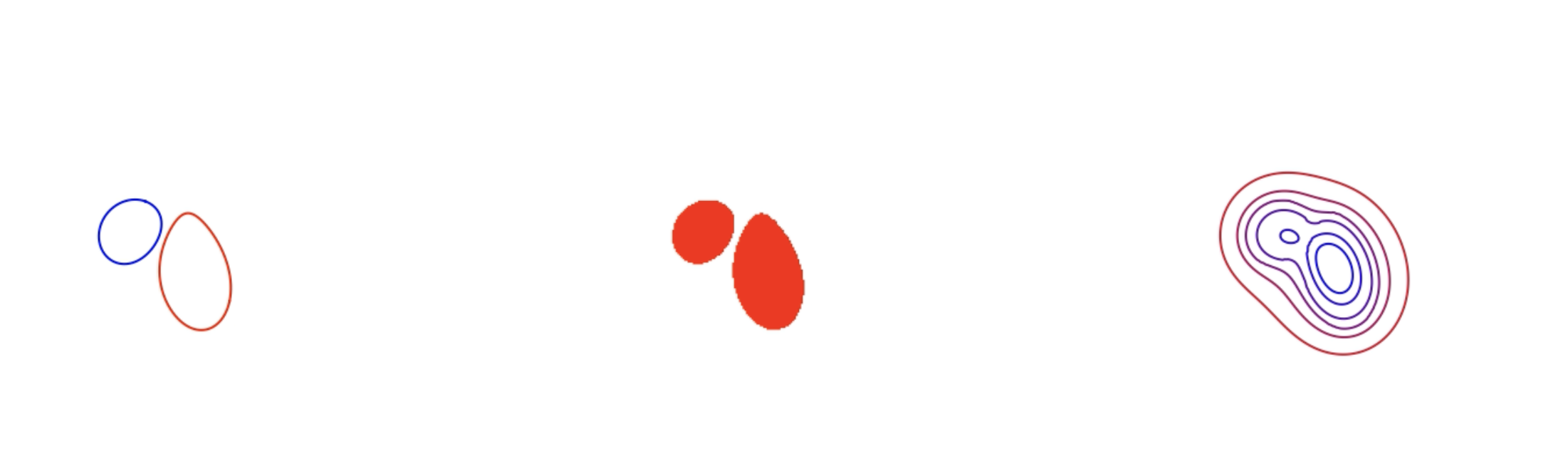}};

    \node (LO10) at ([xshift=-0.1mm]O10T0.west) {\rotatebox[origin=c]{90}{$t=0$}};
    \node (LO20) at ([xshift=-0.1mm]O20T0.west) {\rotatebox[origin=c]{90}{$t=5$}};
    \node (LO30) at ([xshift=-0.1mm]O30T0.west) {\rotatebox[origin=c]{90}{$t=10$}};
    \node (LO40) at ([xshift=-0.1mm]O40T0.west) {\rotatebox[origin=c]{90}{$t=15$}};
    \node (LO50) at ([xshift=-0.1mm]O50T0.west) {\rotatebox[origin=c]{90}{$t=20$}};

    \node[anchor=south, xshift=0mm, yshift=1mm] at ($(O10T0.north west)!0.1666!(O10T0.north east)$) {$z_\pm$};
    \node[anchor=south,  xshift=0mm, yshift=1mm] at ($(O10T0.north west)!0.5000!(O10T0.north east)$) {$F$};
    \node[anchor=south, xshift=0mm, yshift=1mm] at ($(O10T0.north west)!0.8333!(O10T0.north east)$) {$\phi$};
    
    \end{tikzpicture}}
    \caption{Left-handed screened system with data given in Section \ref{sec:lefthanded}}
    \label{fig:LHS_R1}
\end{figure}

\newpage
\vspace{1em}
\noindent{\bf{Acknowledgments.}} The work of P.C. was partially supported by NSF grant DMS-2606072. The work of Z.H. was partly supported by a Simons Collaboration Grant 601960. H.G. acknowledges support from the AMS--Simons Travel Grant and the hospitality of Vincent Martinez at CUNY--Hunter College during the summer of 2026.\\

\noindent{\bf{Disclosure on AI usage.}} 
Generative artificial intelligence tools were used during
manuscript preparation exclusively for literature search, software development, typesetting, and proofreading prose and  algebra. Every suggestion, computation, and
reference produced by the AI systems was independently verified and, where necessary, corrected
or reformulated by the authors, who take full responsibility for the manuscript as it is presently
written.

\end{document}